\documentclass[11pt]{amsart}
\usepackage[nobysame,abbrev,alphabetic]{amsrefs}
\usepackage{amssymb}
\usepackage[all]{xy}

\DeclareMathOperator{\Aut}{Aut}
\DeclareMathOperator{\Char}{char}
\DeclareMathOperator{\cd}{cd}

\DeclareMathOperator{\Frat}{Frat}
\DeclareMathOperator{\Gal}{Gal}

\DeclareMathOperator{\Img}{Im}
\DeclareMathOperator{\Inf}{Inf}
\DeclareMathOperator{\invlim}{\varprojlim}
\DeclareMathOperator{\Ker}{Ker}

\DeclareMathOperator{\Res}{Res}

\DeclareMathOperator{\trg}{trg}

\DeclareFontFamily{U}{wncy}{}
\DeclareFontShape{U}{wncy}{m}{n}{<->wncyr10}{}
\DeclareSymbolFont{mcy}{U}{wncy}{m}{n}
\DeclareMathSymbol{\Sha}{\mathord}{mcy}{"58}
\DeclareMathSymbol{\sha}{\mathord}{mcy}{"78}

\begin{document}

\newtheorem{thm}{Theorem}[section]
\newtheorem{cor}[thm]{Corollary}
\newtheorem{lem}[thm]{Lemma}
\newtheorem{prop}[thm]{Proposition}
\newtheorem{defin}[thm]{Definition}
\newtheorem{exam}[thm]{Example}
\newtheorem{rem}[thm]{Remark}
\newtheorem{case}{\sl Case}
\newtheorem{claim}{Claim}
\newtheorem{prt}{Part}
\newtheorem*{mainthm}{Main Theorem}
\newtheorem*{thmA}{Theorem A}
\newtheorem*{thmB}{Theorem B}
\newtheorem*{thmC}{Theorem C}
\newtheorem*{thmD}{Theorem D}
\newtheorem{question}[thm]{Question}
\newtheorem*{notation}{Notation}
\newtheorem*{acknowledgment}{Acknowledgment}
\swapnumbers
\newtheorem{rems}[thm]{Remarks}
\newtheorem{examples}[thm]{Examples}

\newtheorem{questions}[thm]{Questions}
\numberwithin{equation}{section}

\newcommand{\Bock}{\mathrm{Bock}}
\newcommand{\cupdot}{\mathbin{\mathaccent\cdot\cup}}
\newcommand{\dec}{\mathrm{dec}}
\newcommand{\diam}{\mathrm{diam}}
\newcommand{\dirlim}{\varinjlim}
\newcommand{\discup}{\ \ensuremath{\mathaccent\cdot\cup}}
\newcommand{\divis}{\mathrm{div}}
\newcommand{\gr}{\mathrm{gr}}
\newcommand{\nek}{,\ldots,}
\newcommand{\ind}{\hbox{ind}}
\newcommand{\inv}{^{-1}}
\newcommand{\isom}{\cong}
\newcommand{\Massey}{\mathrm{Massey}}
\newcommand{\ndiv}{\hbox{$\,\not|\,$}}
\newcommand{\nil}{\mathrm{nil}}
\newcommand{\pr}{\mathrm{pr}}
\newcommand{\sep}{\mathrm{sep}}
\newcommand{\tagg}{^{''}}
\newcommand{\tensor}{\otimes}
\newcommand{\alp}{\alpha}
\newcommand{\gam}{\gamma}
\newcommand{\Gam}{\Gamma}
\newcommand{\del}{\delta}
\newcommand{\Del}{\Delta}
\newcommand{\eps}{\epsilon}
\newcommand{\lam}{\lambda}
\newcommand{\Lam}{\Lambda}
\newcommand{\sig}{\sigma}
\newcommand{\Sig}{\Sigma}
\newcommand{\bfA}{\mathbf{A}}
\newcommand{\bfB}{\mathbf{B}}
\newcommand{\bfC}{\mathbf{C}}
\newcommand{\bfF}{\mathbf{F}}
\newcommand{\bfP}{\mathbf{P}}
\newcommand{\bfQ}{\mathbf{Q}}
\newcommand{\bfR}{\mathbf{R}}
\newcommand{\bfS}{\mathbf{S}}
\newcommand{\bfT}{\mathbf{T}}
\newcommand{\bfZ}{\mathbf{Z}}
\newcommand{\dbA}{\mathbb{A}}
\newcommand{\dbC}{\mathbb{C}}
\newcommand{\dbF}{\mathbb{F}}
\newcommand{\dbN}{\mathbb{N}}
\newcommand{\dbQ}{\mathbb{Q}}
\newcommand{\dbR}{\mathbb{R}}
\newcommand{\dbU}{\mathbb{U}}
\newcommand{\dbZ}{\mathbb{Z}}
\newcommand{\grf}{\mathfrak{f}}
\newcommand{\gra}{\mathfrak{a}}
\newcommand{\grA}{\mathfrak{A}}
\newcommand{\grB}{\mathfrak{B}}
\newcommand{\grh}{\mathfrak{h}}
\newcommand{\grI}{\mathfrak{I}}
\newcommand{\grL}{\mathfrak{L}}
\newcommand{\grm}{\mathfrak{m}}
\newcommand{\grp}{\mathfrak{p}}
\newcommand{\grq}{\mathfrak{q}}
\newcommand{\grR}{\mathfrak{R}}
\newcommand{\calA}{\mathcal{A}}
\newcommand{\calB}{\mathcal{B}}
\newcommand{\calC}{\mathcal{C}}
\newcommand{\calE}{\mathcal{E}}
\newcommand{\calG}{\mathcal{G}}
\newcommand{\calH}{\mathcal{H}}
\newcommand{\calK}{\mathcal{K}}
\newcommand{\calL}{\mathcal{L}}
\newcommand{\calS}{\mathcal{S}}
\newcommand{\calW}{\mathcal{W}}
\newcommand{\calV}{\mathcal{V}}

\title[The Kummerian property]{The Kummerian Property and Maximal Pro-$p$ Galois Groups}
\author{Ido Efrat}
\author{Claudio Quadrelli}
\address{Department of Mathematics, Ben-Gurion University of the Negev, P.O.\  Box 653, Be'er-Sheva 84105, Israel}
\email{efrat@math.bgu.ac.il}
\address{Department of Mathematics and Applications, University of Milano-Bicocca,
Via R. Cozzi 55 -- U5, 20125 Milan, Italy}
\email{claudio.quadrelli@unimib.it}


\maketitle


\begin{abstract}
For a  prime number $p$, we give a new restriction on pro-$p$ groups $G$ which are realizable as the maximal pro-$p$ Galois group $G_F(p)$ for a field $F$ containing a root of unity of order $p$.
This restriction arises from Kummer Theory and the structure of the maximal $p$-radical extension of  $F$.
We study it in the abstract context of pro-$p$ groups $G$  with a continuous homomorphism $\theta\colon G\to1+p\dbZ_p$, and characterize it cohomologically, and in terms of 1-cocycles on $G$.
This is used to produce new examples of pro-$p$ groups which do not occur as maximal pro-$p$ Galois groups of fields as above.
\end{abstract}

\section{Introduction}
A major open problem in modern Galois theory is to characterize the profinite groups which are realizable as absolute Galois groups of fields.
A seemingly more approachable problem is to characterize, for a prime number $p$, the pro-$p$ groups $G$ which are realizable as the maximal pro-$p$ Galois group $G_F(p)$ of some field $F$ containing a root of unity of order $p$.

Among the known group-theoretic restrictions on pro-$p$ groups $G$ realizable as $G_F(p)$, with $F$ as above, one has Becker's  pro-$p$ version of the classical Artin--Schreier theorem:
Its finite subgroups can only be trivial or of order $2$ \cite{Becker74}.

Next, it follows from the deep results of Voevodsky and Rost
that the cohomology ring $H^*(G,\dbZ/p)$ is {\sl quadratic}, i.e., it is generated by degree $1$ elements and its relations originate from the degree $2$ part; see \S\ref{sec:not Galois}.

A more recent restriction concerns the {\sl external} cohomological structure of $G$.
Namely, for every $\varphi_1,\varphi_2,\varphi_3\in H^1(G,\dbZ/p)$, the 3-fold Massey product $\langle\varphi_1,\varphi_2,\varphi_3\rangle$ is not essential (\cite{Matzri14}, \cite{EfratMatzri16}, \cite{MinacTan16b}; see \S\ref{section on Massey products} for these notions).

In the present paper we give a new restriction on pro-$p$ groups $G$ which are realizable as $G_F(p)$, for a field $F$ containing a root of unity of order $p$.
This restriction arises from Kummer theory, specifically, from the structure of the maximal $p$-radical extension $F(\root{p^\infty}\of F)$ of $F$.
We apply this property to give new explicit examples of groups which are not realizable in this way.

More specifically, we study a property of pairs $\calG=(G,\theta)$, where $\theta$ is a continuous homomorphism from $G$ to the group $1+p\dbZ_p$ of the 1-units in $\dbZ_p$.
We assume for the moment that $\Img(\theta)$ is a torsion free subgroup of $1+p\dbZ_p$, which is always the case for $p$ odd.
For such a pair $\calG$ we define a certain canonical closed normal subgroup $K(\calG)$ of $G$, and observe that the quotient $G/K(\calG)$ decomposes as a semi-direct product $A\rtimes\bar G$, where $A$ is an abelian pro-$p$ group, $\bar G$ is isomorphic to either $\dbZ_p$ or $\{1\}$, and the action is by exponentiation by $\theta$ (see Proposition \ref{prop: calG/K}(a)).
We call $\calG$ {\sl Kummerian} if, moreover, $A$ is a free abelian pro-$p$ group.
The motivating example for this notion is Galois theoretic: For a field $F$ as above,  we let $\theta=\theta_{F,p}\colon G=G_F(p)\to1+p\dbZ_p$ be the pro-$p$ cyclotomic character.
Then Kummer theory implies that $\calG_F=(G,\theta)$ is Kummerian  (see \S\ref{sec:The Galois case} for details).
Note that in this situation the assumption that $\Img(\theta_{F,p})$ is torsion free just means that $\sqrt{-1}\in F$ if $p=2$.

The Kummerian property is intimately related to {\sl 1-cocycles}.
For example, assuming that $G$ is finitely generated, the pair $\calG=(G,\theta)$ is Kummerian if and only if $K(\calG)=\bigcap_cc\inv(\{0\})$, where the intersection is over all 1-cocycles $c\colon G\to\dbZ_p(1)_\theta$, and $\dbZ_p(1)_\theta$ is the $G$-module with underlying group $\dbZ_p$ and $G$-action induced by $\theta$ (see Theorem \ref{thm:K intersection}).
Moreover, it follows from results of Labute \cite{Labute67a} that $\calG$ is Kummerian if and only if every map $\alp\colon X\to\dbZ_p$, where $X$ is a minimal generating set of $G$, extends to a 1-cocycle $c\colon G\to\dbZ_p(1)_\theta$.
It is this latter lifting property which we use in \S\ref{sec:not Galois} to rule out various pro-$p$ groups $G$ from being maximal pro-$p$ Galois groups $G_F(p)$ as above.

Our new restriction implies the Artin--Schreier/Becker restriction on the finite subgroups of $G_F(p)$ (Remark \ref{rem:positselski}(3)), but is independent of the cohomological quadraticness restriction, as well as the  restriction on 3-fold Massey products (see \S\ref{section on Massey products}, as well as Examples \ref{exam:commutators1} and \ref{exam:commutators2}).

We further show that the class of Kummerian cyclotomic pro-$p$ pairs is closed under basic operations:
free products, and extensions by a free abelian pro-$p$ group (see Propositions \ref{prop:freeprod} and \ref{prop:semidirectprod}).

\medskip

Some of the results of this paper were presented at the BIRS workshop on ``Nilpotent Fundamental Groups" on June 2017.
Following it we were informed by J\'an Min\'a\v c about the (then) forthcoming work \cite{MinacRogelstadTan18}  on the structure of relations in $G_F(p)$, which is related to these results.

\medskip

We warmly thank Nguyen Duy T\^an for pointing out to the second-named author the possible importance of \cite{Labute67a}*{Prop.\ 6} for results as in our \S\ref{sec:not Galois}.
We thank J\'an Min\'a\v c and Thomas Weigel (the former thesis advisors of the second-named author) for other inspiring discussions on related topics.
This research was supported by the Israel Science Foundation (grant No.\ 152/13).
The second-named author was also partially supported by the BGU Center for Advanced Studies in Mathematics.


\section{Cyclotomic pro-$p$ pairs}
\label{section on cyclotomic pairs}

Let $p$ be a fixed prime number.
Let $1+p\dbZ_p$ be the group of 1-units in $\dbZ_p$.
Thus $1+p\dbZ_p\isom\dbZ_p$ for $p\neq2$, and $1+2\dbZ_2\isom(\dbZ/2)\oplus\dbZ_2$.
Following \cite{Efrat95} and \cite{Efrat98}, we define a {\sl cyclotomic pro-$p$ pair} to be a pair $\calG=(G,\theta)$ consisting of a pro-$p$ group $G$ and a continuous homomorphism $\theta\colon G\to1+p\dbZ_p$.
We say that $\calG$ is {\sl finitely generated} if $G$ is finitely generated as a pro-$p$ group.
We say that $\calG$ is {\it torsion free} if $\Img(\theta)$ is a torsion free subgroup of $1+p\dbZ_p$.
Note that this is always the case when $p\neq2$.

A {\sl morphism} $\varphi\colon (G_1,\theta_1)\to(G_2,\theta_2)$ of cyclotomic pro-$p$ pairs is a continuous homomorphism $\varphi\colon G_1\to G_2$ of the pro-$p$ groups such that $\theta_1=\theta_2\circ\varphi$.
We say that $\varphi$ is a {\sl cover} if $\varphi$ induces an isomorphism $G_1/\Frat(G_1)\to G_2/\Frat(G_2)$, where $\Frat(G_i)=G_i^p[G_i,G_i]$ denotes the Frattini subgroup of $G_i$.

\smallskip

We list some basic constructions in the category of cyclotomic pro-$p$ pairs.

\medskip

(1) \quad
For a cyclotomic pro-$p$ pair $\calG=(G,\theta)$ and a closed normal subgroup $N$ of $G$ contained in $\Ker(\theta)$ we define the {\sl quotient} $\calG/N$ to be the pair $(G/N,\bar\theta)$, where $\bar\theta\colon G/N\to1+p\dbZ_p$ is the homomorphism induced by $\theta$.

\medskip

(2) \quad
Given cyclotomic pro-$p$ pairs $\calG_1=(G_1,\theta_1)$ and $\calG_2=(G_2,\theta_2)$, the {\sl free product} $\calG_1*\calG_2$ is the pair $(G,\theta)$, where $G=G_1*_pG_2$ is the free product of $G_1,G_2$ in the category of pro-$p$ groups, and $\theta\colon G\to1+p\dbZ_p$ is the unique continuous homomorphism extending $\theta_1$ and $\theta_2$, given by the universal property of $G$.

\medskip

(3) \quad
Consider a cyclotomic pro-$p$ pair $\bar\calG=(\bar G,\bar\theta)$ and an abelian pro-$p$ group $A$, written multiplicatively.
We define the {\sl extension} $A\rtimes\bar\calG$ to be the pair $(G,\theta)$, where $G=A\rtimes\bar G$ with the action given by ${}^{\bar g}h=h^{\bar\theta(\bar g)}$ for $h\in A$ and $\bar g\in\bar G$, and where $\theta\colon A\rtimes\bar G\to1+p\dbZ_p$ is the composition of the projection $G\to\bar G$ and $\bar\theta$.
The Frattini quotient of $G=A\rtimes\bar G$ is
\begin{equation}
\label{eq:Frat quotient of extension}
G/\Frat(G)=(A/A^p)\times(\bar G/\Frat(\bar G)).
\end{equation}

\begin{lem}
\label{lem:finitely generated extension}
In the above setup,  $G$ is finitely generated if and only if both $A$ and $\bar G$ are finitely generated (as pro-$p$ groups).
\end{lem}
\begin{proof}
This follows from (\ref{eq:Frat quotient of extension}) and the Frattini argument \cite{NeukirchSchmidtWingberg}*{Prop.\ 3.9.1}.
\end{proof}

Note that $A'\rtimes(A\rtimes\bar\calG)\isom(A\times A')\rtimes\bar\calG$ for any abelian pro-$p$ groups $A,A'$.

Given a morphism $\bar\varphi\colon\bar\calG'\to\bar\calG$ of cyclotomic pro-$p$ pairs and a continuous homomorphism $A'\to A$ of pro-$p$ abelian groups, there is an induced morphism $\varphi\colon A'\rtimes\bar\calG'\to A\rtimes\bar\calG$.
Using (\ref{eq:Frat quotient of extension}) we obtain:

\begin{lem}
\label{lem:covers of extensions}
In the above setup, $\varphi$ is a cover if and only if $\bar\varphi$ is a cover and the induced map $A'/(A')^p\to A/A^p$ is an isomorphism.
\end{lem}

Finally, we will make repeated use of the following observation on abelian groups.

\begin{lem}
\label{lem:Frattini quotients and isomorphisms}
Let $\varphi\colon B\to A$ be a continuous homomorphism of abelian pro-$p$ groups.
Suppose that $\varphi$ induces an isomorphism $B/B^p\to A/A^p$ of the Frattini quotients, and that $A$ is a free abelian pro-$p$ group.
Then $\varphi$ is an isomorphism.
\end{lem}


\section{The subgroup $K(\calG)$}
\label{sec:thetacomm}

Given a cyclotomic pro-$p$ pair $\calG=(G,\theta)$, let
\[
\label{eq:thetacomm}
 K(\calG)=\Bigl\langle h^{-\theta(g)}ghg\inv \Bigm| g\in G, \ h\in\Ker(\theta)\Bigr\rangle.
\]
Thus $K(\calG)$ is the closed subgroup of $G$ with generators as stated.
Note that $K(\calG)$ is a normal subgroup of $G$, and $K(\calG)\leq\Ker(\theta)$ with $\Ker(\theta)/K(\calG)$ abelian.
In particular, we have the quotient $\calG/K(\calG)$, in the sense of \S\ref{section on cyclotomic pairs}.
If $\theta=1$, then $K(\calG)=[G,G]$ is the commutator (closed) subgroup of $G$.

Since $\Img(\theta)\subseteq 1+p\dbZ_p$, the generators $ h^{-\theta(g)}ghg\inv$ of $K(\calG)$ belong to $\Frat(G)=G^p[G,G]$,
so $K(\calG)\subseteq\Frat(G)$.
Therefore the morphism $\calG\to\calG/K(\calG)$ is a cover.

The map  $\calG\mapsto K(\calG)$ is a functor from the category of cyclotomic pro-$p$ pairs to the category of pro-$p$ groups.
The map $\calG\mapsto\calG/K(\calG)$ is a functor from the category of cyclotomic pro-$p$ pairs to itself.

\begin{exam}
\label{exam:K of Zp}
\rm
Let $\calG=(G,\theta)$ be a cyclotomic pro-$p$ pair with $G\isom\dbZ_p$.
Then $K(\calG)$ is trivial.
Indeed, if $\theta=1$, then $K(\calG)=[G,G]=\{1\}$, and else $\Ker(\theta)=\{1\}$, so all the generators of $K(\calG)$ are $1$.
\end{exam}

\begin{lem}
\label{lem:K of an extension}
Let $\bar\calG=(\bar G,\bar\theta)$ be a cyclotomic pro-$p$ pair, let $A$ be an abelian pro-$p$ group, and set $\calG=A\rtimes\bar\calG$.
Then $K(\calG)=\{1\}\times K(\bar\calG)$.
\end{lem}
\begin{proof}
We write $A$ multiplicatively and set $\calG=(G,\theta)$.
Then $\Ker(\theta)=A\times\Ker(\bar\theta)$.

We show that the generators of $K(\calG)$ are the same as the generators of $K(\bar\calG)$, when one identifies $\bar G$ as a subgroup of $G$.
Let $g\in G$ and $h\in \Ker(\theta)$.
 We may write $g=a\bar g$ and $h=a'\bar h$ with $a,a'\in A$, $\bar g\in\bar G$ and $\bar h\in\Ker(\bar\theta)$.
Then $\theta(g)=\bar\theta(\bar g)$ and  $\bar ga'=(a')^{\bar\theta(\bar g)}\bar g$.
Also $a$ commutes with both $a'$ and $\bar g\bar h\bar g\inv$, and $a'$ commutes with $\bar h$.
We now compute:
\[
\begin{split}
h^{-\theta(g)}ghg\inv&=h^{-\theta(g)}a\bar ga'\bar h\bar g\inv a\inv
=h^{-\theta(g)}a(a')^{\bar\theta(\bar g)}(\bar g\bar h\bar g\inv)a\inv \\
&=(a'\bar h)^{-\bar\theta(\bar g)}(a')^{\bar\theta(\bar g)}\bar g\bar h\bar g\inv =\bar h^{-\bar\theta(\bar g)}\bar g\bar h\bar g\inv,
\end{split}
\]
as desired.
\end{proof}

The subgroup $K(\calG)$ is characterized by the following minimality property:

\begin{prop}
\label{prop: calG/K}
\begin{enumerate}
\item[(a)]
When $\calG$ is torsion free,
\[
\calG/K(\calG)\isom(\Ker(\theta)/K(\calG))\rtimes(\calG/\Ker(\theta)).
\]
In particular, $K(\calG/K(\calG))=\{1\}$.
\item[(b)]
Let $\bar\calG=(\bar G,\bar\theta)$ be a cyclotomic pro-$p$ pair such that $\Ker(\bar\theta)=\{1\}$, and let $A$ be an abelian pro-$p$ group.
Then every morphism $\varphi\colon\calG\to A\rtimes\bar\calG$ of cyclotomic pro-$p$ pairs factors through $\calG/K(\calG)$.
\end{enumerate}
\end{prop}
\begin{proof}
(a) \quad
Set $\calG=(G,\theta)$.
Since it is torsion free,  $G/\Ker(\theta)\isom\Img(\theta)$ is isomorphic to either $\dbZ_p$ or $\{1\}$.
In particular, the epimorphism $G/K(\calG)\to G/\Ker(\theta)$ splits.
Hence
\[
G/K(\calG)\isom(\Ker(\theta)/K(\calG))\rtimes(G/\Ker(\theta)),
 \]
where the action is given by $\bar g\bar h\bar g\inv=\bar h^{\theta(g)}$ for a coset  $\bar h$  of $h\in\Ker(\theta)$  in $\Ker(\theta)/K(\calG)$ and a coset $\bar g$ of $g\in G$ in $G/\Ker(\theta)$.
The first assertion follows.

The second assertion of (a)  follows from the first one and from Lemma \ref{lem:K of an extension}.

\medskip

(b) \quad
By the functoriality of $K$ and by Lemma \ref{lem:K of an extension},
\[
\varphi(K(\calG))\subseteq K(A\rtimes\bar \calG)=\{1\}\times K(\bar\calG)=\{1\}.
\qedhere
\]
\end{proof}

Motivated by Theorem \ref{thm:Kummer example} below, which is a consequence of Kummer theory, we define the following key notion:

\begin{defin}
\rm
A cyclotomic pro-$p$ pair $\calG=(G,\theta)$ is {\sl Kummerian} if $\Ker(\theta)/K(\calG)$ is a free abelian pro-$p$ group.
\end{defin}

Recall that this quotient is always an abelian pro-$p$ group.

\begin{examples}
\label{exam:Kummerian with trivial theta}
\rm
(1)  \quad
Given a pro-$p$ group $G$, the cyclotomic pair $(G,1)$ is Kummerian if and only if the abelianization
$G^{\rm ab}=G/[G,G]$ is a free abelian pro-$p$ group.

\medskip

(2) \quad
In particular,  a cyclotomic pro-$p$ pair of the form $(\dbZ/p,\theta)$ cannot be Kummerian when $p\neq2$, since necessarily $\theta=1$.

\medskip

(3) \quad
By contrast, when $p=2$ the pair $\calG=(\dbZ/2,\theta)$, where $\theta$ is the nontrivial homomorphism $\dbZ/2\to1+2\dbZ/2$, is Kummerian.
Indeed, here one has $\Ker(\theta)=K(\calG)=\{1\}$.

\medskip

(4) \quad
When $p=2$, a cyclotomic pro-$p$ pair of the form $\calG=(\dbZ/4,\theta)$ cannot be Kummerian.
Indeed, when $\theta=1$ this follows from (1).
Otherwise $\Ker(\theta)=2\dbZ/4\dbZ$, and $K(\calG)=\{0\}$, so $\Ker(\theta)/K(\calG)\isom\dbZ/2$.

\medskip

(5) \quad
Still in the case where $p=2$, a cyclotomic pro-$p$ pair  $\calG=(G,\theta)$ with $G\isom(\dbZ/2)^2$ cannot be Kummerian.
Indeed,  $\theta(g)=\pm1$ for every $g\in G$, so $K(\calG)=\{1\}$.
Moreover, since $G$ does not embed in $1+2\dbZ_2$, the kernel $\Ker(\theta)$ is nontrivial.
Being finite, it is therefore not a free abelian pro-$2$ group.

\medskip

(6) \quad A cyclotomic pro-$p$ pair $\calG=(G,\theta)$ is Kummerian if and only if $\calG/K(\calG)=(G/K(\calG),\bar\theta)$ is Kummerian.
Indeed, $\Ker(\theta)/K(\calG)=\Ker(\bar\theta)$, and by Proposition \ref{prop: calG/K}(a), $K(\calG/K(\calG))=\{1\}$.
\end{examples}

\begin{prop}
\label{prop:semidirectprod}
Let $\bar\calG$ be a cyclotomic pro-$p$ pair and let $A$ be a free abelian pro-$p$ group.
Then $\calG=A\rtimes\bar\calG$ is Kummerian if and only if $\bar\calG$ is Kummerian.
\end{prop}
\begin{proof}
Let $\bar\calG=(\bar G,\bar\theta)$ and $\calG=(A\rtimes\bar G,\theta)$.
Then $\Ker(\theta)=A\times\Ker(\bar\theta)$.
By Lemma \ref{lem:K of an extension}, $K(\calG)=\{1\}\times K(\bar\calG)$.
Hence
\[
\Ker(\theta)/K(\calG)=A\times(\Ker(\bar\theta)/K(\bar\calG)),
\]
and the desired equivalence follows.
\end{proof}

\begin{cor}
\label{cor:K trivial}
Let $\calG=(G,\theta)$ be a torsion free cyclotomic pro-$p$ pair, with $G\neq1$.
The following conditions are equivalent:
\begin{enumerate}
\item[(a)]
$\calG$ is Kummerian and $K(\calG)=\{1\}$;
\item[(b)]
$\calG\isom A\rtimes\bar\calG$ for some free abelian pro-$p$ group $A$ and some cyclotomic pro-$p$ pair $\bar\calG=(\bar G,\bar\theta)$ with $\bar G\isom\dbZ_p$.
\end{enumerate}
\end{cor}
\begin{proof}
(a)$\Rightarrow$(b): \quad
By Proposition \ref{prop: calG/K}(a),
$\calG\isom A\rtimes(\calG/\Ker(\theta))$ for a free abelian pro-$p$ group $A$.
As before, either $G/\Ker(\theta)\isom\Img(\theta)\isom\dbZ_p$ or else $\theta=1$.
In the first case we are done.

In the second case, as $G=A\neq1$,  we may write $A=A'\times\dbZ_p$ for some free abelian pro-$p$ group $A'$.
Then $\calG\isom A'\rtimes(\dbZ_p,1)$.

\medskip

(b)$\Rightarrow$(a): \quad
The pair $\bar\calG$ is Kummerian, so by Proposition \ref{prop:semidirectprod}, $A\rtimes\bar\calG$ is also Kummerian.
By Example \ref{exam:K of Zp} and Lemma \ref{lem:K of an extension}, $K(A\rtimes\bar\calG)=\{1\}$.
\end{proof}

\section{The Galois case}
\label{sec:The Galois case}
Throughout this section let $F$ be a field containing a root of unity of order $p$ (in particular, $\Char(F)\neq p$).
Let $F(p)$ be the compositum of all finite Galois $p$-extensions of $F$, so $G_F(p)=\Gal(F(p)/F)$ is the maximal pro-$p$ Galois group of $F$.
Let $\mu_{p^n}$ be the group of all roots of unity of order dividing $p^n$ in the algebraic closure of $F$, and set $\mu_{p^\infty}=\bigcup_{n=1}^\infty\mu_{p^n}$.
In fact, $\mu_{p^\infty}\subseteq F(p)$.
The group $\Aut(\mu_{p^\infty}/\mu_p)$ of all automorphisms of $\mu_{p^\infty}$ fixing $\mu_p$ is isomorphic to $1+p\dbZ_p$;
More specifically, $\sig\in\Aut(\mu_{p^\infty}/\mu_p)$ corresponds to the $p$-adic unit $\lam\in1+p\dbZ_p$ such that $\sig(\zeta)=\zeta^\lam$ for every $\zeta\in\mu_{p^\infty}$.
The composition of the restriction map $G_F(p)\to\Aut(\mu_{p^\infty}/\mu_p)$ and this isomorphism is the {\sl pro-$p$ cyclotomic character}
\[
\theta_{F,p}\colon G_F(p)\longrightarrow1+p\dbZ_p.
\]

\begin{defin}
\label{def:Galois}
\rm
The {\sl cyclotomic pro-$p$ pair of the field $F$} is the pair
\[
\calG_F=(G_F(p),\theta_{F,p}).
\]
It is torsion free if and only if $p\neq2$ or else $p=2$ and $\sqrt{-1}\in F$.

Furthermore, for every Galois extension $E/F$ such that $F(\mu_{p^\infty})\subseteq E\subseteq F(p)$, one has $E(p)=F(p)$ and $G_E(p)\leq\Ker(\theta_{F,p})$.
We define the cyclotomic pro-$p$ pair of $E/F$ to be
\[
\calG(E/F)=\calG_F/G_E(p)=(\Gal(E/F),\theta_{E/F,p}),
\]
where $\theta_{E/F,p}\colon\Gal(E/F)\to1+p\dbZ_p$ is the homomorphism induced by $\theta_{F,p}$
\end{defin}

In the next theorem we restrict ourselves to torsion free pairs $\calG_F$.
Then the subgroup $K(\calG_F)$ can be naturally interpreted as in part (d) of the Theorem.
This connection was first observed and is studied in the forthcoming paper by T.\ Weigel and the second-named author \cite{QuadrelliWeigel}.

\begin{thm}
\label{thm:Kummer example}
Assume that $\sqrt{-1}$ if $p=2$ and let  $E=F(\sqrt[p^\infty]{F})$ be the field obtained by adjoining to $F$ all roots of $p$-power degree of elements of $F$.
Also let
\[
A=\Gal(E/F(\mu_{p^\infty}))=\Ker(\theta_{F,p})/G_E(p)
\]
\[
\bar\calG=\calG(F(\mu_{p^\infty})/F)=\calG_F/\Ker(\theta_{F,p}).
\]
Then:
\begin{enumerate}
\item[(a)]
$\bar\calG=\calG(E/F)/\Ker(\theta_{E/F,p})$.
\item[(b)]
$\calG(E/F)=A\rtimes\bar\calG$.
\item[(c)]
$A$ is a free abelian pro-$p$ group.
\item[(d)]
 $K(\calG_F)=G_E(p)$.
 \item[(e)]
 $\calG_E=(K(\calG_F),1)$.
 \item[(f)]
$\calG_F$ is Kummerian.
\end{enumerate}
\end{thm}
\begin{proof}
(a) \quad
Trivial.

\medskip

(b) \quad
The assumptions imply that $\Gal(F(\mu_{p^\infty})/F)$ is isomorphic to either $\dbZ_p$ or $1$.
Hence there is a natural semi-direct product decomposition
\[
\Gal(E/F)= A\rtimes\Gal(F(\mu_{p^\infty})/F),
\]

To compute the action, let $\sig\in\Gal(E/F)$ and $\tau\in A$.
It suffices to show that $(\sig\tau\sig\inv)(\root q\of a)=\tau^{\theta(\sig)}(\root q\of a)$ for every  $p$-power $q$ and a $q$-th root $\root q \of a$ of an element $a$ of $F^\times$, where we abbreviate $\theta=\theta_{E/F,p}$.
We may write $\sig(\root q\of a)=\zeta\root q\of a$ and $\tau(\root q\of a)=\omega\root q\of a$ for some $q$-th roots of unity $\zeta,\omega$.
Then $\sig\inv(\root q\of a)=\zeta^{-\theta(\sig\inv)}\root q\of a$, so
$(\tau\sig\inv)(\root q\of a)=\zeta^{-\theta(\sig\inv)}\omega\root q\of a$.
This implies that
\[
(\sig\tau\sig\inv)(\root q\of a)=(\zeta^{\theta(\sig)})^{-\theta(\sig\inv)}\omega^{\theta(\sig)}\zeta\root q\of a=\omega^{\theta(\sig)}\root q\of a=\tau^{\theta(\sig)}(\root q\of a).
\]
Thus $\sig\tau\sig\inv=\tau^{\theta(\sig)}$, as required.

\medskip
(c) \quad
This seems well known, but we provide a proof due to lack of reference.
Set $L=F(\mu_{p^\infty})$ and for $n\geq1$ let $T_n=F^\times/(F^\times\cap (L^\times)^{p^n})$.
Note that  $T_n$ is a $p^n$-torsion (discrete) abelian group.

Now Kummer theory \cite[Ch.\  VI, Th.\  8.1]{Lang02} gives for every $n$ a
commutative diagram of non-degenerate bilinear maps
\[
\xymatrix{
\Gal (L(\sqrt[p^{n+1}]{F})/L)\ar[d] &*-<3pc>{\times } & T_{n+1}\ar[d]\ar[r]^{ (\cdot ,\cdot )_{n+1}} &\mu _{p^{n+1}}\ar[d]^{p}\\
\Gal (L(\sqrt[p^{n}]{F})/L) & *-<3pc>{\times } & \strut T_{n}\ar[r]^{(\cdot ,\cdot )_{n}} &\strut \mu _{p^{n}},
}
\]
where $(\sigma  ,\bar a)_{n}=\sigma  (\sqrt[p^{n}]{a})/\sqrt[p^{n}]{a}$, and
similarly for $(\cdot ,\cdot )_{n+1}$. It follows that
\[
\Gal (E/L)=\invlim \Gal (L(\sqrt[p^{n}]{F})/L)
\]
is a torsion-free abelian pro-$p$ group, whence a free abelian pro-$p$ group.


(d) \quad
Write $\bar\calG=(\bar G,\bar\theta)$.
Thus $\bar G=\Gal(F(\mu_{p^\infty})/F)$ and $\Ker(\bar\theta)=\{1\}$.
By Proposition \ref{prop: calG/K}(a), $\calG_F/K(\calG_F)\isom B\rtimes\bar\calG$, where $B=\Ker(\theta_{F,p})/K(\calG_F)$.
By (c) and Proposition \ref{prop: calG/K}(b), the canonical morphism $\calG_F\to\calG(E/F)$ factors via $\calG_F/K(\calG_F)$.
We obtain a commutative square of morphisms
\[
\xymatrix{
\calG_F/K(\calG_F)\ar[r]\ar[d]_{\wr} & \calG(E/F)\ar[d]^{\wr} \\
B\rtimes\bar\calG\ar[r] & A\rtimes\bar\calG.
}
\]
Since $K(\calG_F)$ and $G_E(p)$ are both contained in $\Frat(G_F(p))$, the upper horizontal epimorphism is a cover.
Therefore so is the lower epimorphism,  and by Lemma \ref{lem:covers of extensions}, the induced map $B/B^p\xrightarrow{\sim} A/A^p$ is an isomorphism.
Since $B$ is an abelian pro-$p$ group and  $A$ is a free abelian pro-$p$ group (by (c)), Lemma \ref{lem:Frattini quotients and isomorphisms} implies that the map $B\to A$ is an isomorphism.
It follows that  $K(\calG_F)=G_E(p)$.

\medskip

(e) \quad
This follows immediately from (d).

\medskip

(f) \quad
By (d), $\Ker(\theta_{F,p})/K(\calG_F)=\Ker(\theta_{F,p})/G_E(p)=A$, and this is a free abelian pro-$p$ group, by (c).
\end{proof}

\begin{rems}
\label{rem:positselski}
\rm
(1) \quad
Let $F$ and $E$ be as in Theorem \ref{thm:Kummer example}.
In \cite{Positselski05} Positselski conjectures that $G_E(p)$ is a free pro-$p$ group.
This is a variant of a conjecture due to Bogomolov \cite{Bogomolov95}, which predicts that for every field $F$ which contains an algebraically closed subfield, the $p$-Sylow subgroup of the commutator  of the absolute Galois group $G_F$ of $F$ is a free pro-$p$ group.

\medskip

(2) \quad
Suppose that $\Char\,F\neq p$ and $F$ contains {\sl all} roots of unity of $p$-power order.
Then $\calG_F=(G_F(p),1)$.
As $\calG_F$ is Kummerian (Theorem \ref{thm:Kummer example}(f)), Example \ref{exam:Kummerian with trivial theta}(1) recovers in this case the well known fact that $G_F(p)^{\rm ab}$ is in this case a free abelian pro-$p$ group.

\medskip

(3) \quad
In view of Theorem \ref{thm:Kummer example}(f) and Examples \ref{exam:Kummerian with trivial theta}(2)(4)(5),
there are no fields $F$ containing a root of unity of order $p$ such that $G_F(p)$ is isomorphic to $\dbZ/p$ with $p$ odd, to $\dbZ/4$, or to $(\dbZ/2)^2$.
Consequently, the only  finite groups of the form $G_F(p)$, with $F$ as above, can be of order $1$ or $2$.
This recovers a result of Becker \cite{Becker74}.
As a special case one recovers the classical Artin--Scherier theorem, asserting that for a field $F$ with separable closure $F_{\rm sep}$, the degree $[F_{\rm sep}:F]$ is either $1$, $2$, or $\infty$.

\medskip

(4) \quad
Let $F$ be a field containing a root on unity of order $p$, and containing $\sqrt{-1}$ if $p=2$.
One says that $F$ is {\sl $p$-rigid} if for every $a,b\in F^\times$  with associated Kummer elements $(a)_F,(b)_F$ in $H^1(G_F(p),\dbZ/p)$, if $(a)_F\cup(b)_F=0$ in $H^2(G_F(p),\dbZ/p)$, then
$(b)_F=(a)_F^i$ for some $0\leq i\leq p-1$, or $(a)_F=0$ \cite{Ware92}.

Suppose that $F^\times/(F^\times)^p$ is finite.
By \cite{CheboluMinacQuadrelli15}*{Cor.\ 3.17}, $\calG_F$ satisfies the equivalent conditions of Corollary \ref{cor:K trivial} if and only if $F$ is a $p$-rigid field.
\end{rems}

\begin{exam}
\label{exam:Iwasawa theory example}
\rm
In the setup of Theorem \ref{thm:Kummer example}, denote $N=G_{F(\mu_{p^\infty})}(p)$, and let $L=F(\mu_{p^\infty})^{p,{\rm ab}}$ be the maximal pro-$p$ abelian extension of $F(\mu_{p^\infty})$, i.e., the fixed field in $F(p)$ of the commutator subgroup $[N,N]$.
Since the extension $E/F(\mu_{p^\infty})$ is pro-$p$ abelian, $E\subseteq L$.
The restrictions induce epimorphisms
\[
\calG_F/K(\calG_F)\to\calG(L/F)/K(\calG(L/F))\to\calG(E/F)/K(\calG(E/F)).
\]
Furthermore, $\calG(E/F)\isom \calG_F/K(\calG_F)$, so by Proposition \ref{prop: calG/K}(a), $K(\calG(E/F))=\{1\}$, implying that the above epimorphisms are injective.
Since $\calG_F$ is Kummerian, we deduce from Example \ref{exam:Kummerian with trivial theta}(6) that
\[
\calG(E/F)=\calG(F(\sqrt[p^\infty]{F})/F),  \qquad \calG(L/F)=\calG(F(\mu_{p^\infty})^{p,{\rm ab}}/F)
\]
are also Kummerian.

When $F$ is a number field, $N$ is a free pro-$p$ group \cite{NeukirchSchmidtWingberg}*{Cor.\ 8.1.18}, so $N/[N,N]$ is a free abelian pro-$p$ group.
Its structure as a $\dbZ_p[\![\Gal(F(\mu_{p^\infty})/F)]\!]$-module is important in the context of Iwasawa theory.
The $\dbZ_p[\![\Gal(F(\mu_{p^\infty})/F)]\!]$-module $N/N^p[N,N]$ is also of importance (see e.g., \cite{BarySorJardenNeftin}), however the corresponding cyclotomic pair $\calG_F/N^p[N,N]$ is not Kummerian.
We thank the referee for pointing out these connections.
\end{exam}

\section{The structure of $\Ker(\theta)/K(\calG)$}
Given closed subgroups $H_1,H_2$ of a pro-$p$ group $G$, we write $[H_1,H_2]$ for the closed subgroup of $G$ generated by all commutators $[h_1,h_2]=h_1\inv h_2\inv h_1h_2$ with $h_1\in H_1$ and $h_2\in H_2$.

\begin{lem}
\label{lem:short exact sequence}
Let $N$ be a closed normal subgroup of a pro-$p$ group $G$ such that $G/N$ is a free pro-$p$ group.
There is a split short exact sequence
\[\xymatrix@C=0.7truecm{
1\ar[r] & N/N^p[G,N]\ar[r] & G/G^p[G,G]\ar[r] & G/NG^p[G,G]\ar[r] &1.}
\]
\end{lem}
\begin{proof}
One has $H^2(G/N,\dbZ/p)=0$ \cite{NeukirchSchmidtWingberg}*{Prop.\ 3.5.17}.
The five term sequence in cohomology \cite{NeukirchSchmidtWingberg}*{Prop.\ 1.6.7} therefore implies that the restriction map
$\Res\colon H^1(G,\dbZ/p)\to H^1(N,\dbZ/p)^G$ is surjective.
Also, the substitution maps give rise to  a commutative diagram of non-degenerate bilinear maps
(see \cite{Efrat Minac11}*{Cor.\ 2.2})
\[
\xymatrix{
G/G^p[G,G]\times H^1(G,\dbZ/p)\ar@<7ex>[d]^{\Res}\ar[r] & \dbZ/p\ar@{=}[d]\\
N/N^p[G,N]\times H^1(N,\dbZ/p)^G\ar@<7ex>[u]\ar[r]&\dbZ/p,
}
\]
where the left vertical map is induced by the inclusion $N\leq G$.
It follows that the left vertical map is injective.
The exactness of the sequence follows.
Since it consists of elementary abelian $p$-groups, it splits.
\end{proof}

\begin{prop}
\label{prop: exact sequence with M and N}
Let $\calG=(G,\theta)$ be a torsion free cyclotomic pro-$p$ pair, and set $N=\Ker(\theta)$.
Then there is a split short exact sequence of elementary abelian $p$-groups
\[\xymatrix@C=0.7truecm{
1\ar[r] &  N/K(\calG)N^p\ar[r] &  G/G^p[G,G]\ar[r] &  G/NG^p\ar[r] & 1.}
\]
\end{prop}
\begin{proof}
As noted earlier, $K(\calG)\leq N$.
Since $\calG$ is torsion free, $G/N\isom\Img(\theta)$ is either $\dbZ_p$ or $\{1\}$.
Lemma \ref{lem:short exact sequence} for  the closed normal subgroup $N/K(\calG)$ of the pro-$p$ group $G/K(\calG)$ yields the exact sequence
\[\xymatrix@C=0.6truecm{
1\ar[r] &  N/K(\calG)N^p[G,N]\ar[r] &  G/K(\calG)G^p[G,G]\ar[r] &  G/NG^p\ar[r] &  1.}
\]
Moreover, for every $g\in G$ and $h\in N$ one has
\[ gh\inv g\inv h=\left(h^{-\theta(g)}ghg\inv\right)\inv\cdot h^{1-\theta(g)}\in K(\calG)N^p. \]
Therefore $[G,N]\leq K(\calG)N^p$.

Also, we have noted that $K(\calG)\leq\Frat(G)=G^p[G,G]$, and the assertion follows.
\end{proof}

\begin{lem}
\label{equal rank}
Let $\pi\colon\calG_1=(G_1,\theta_1)\to\calG_2=(G_2,\theta_2)$ be a cover of torsion free cyclotomic pro-$p$ pairs.
Then $\pi$ induces an epimorphism  $\Ker(\theta_1)/K(\calG_1)\to\Ker(\theta_2)/K(\calG_2)$ of pro-$p$ groups, which is an isomorphism on the Frattini quotients.
\end{lem}
\begin{proof}
For $i=1,2$ we denote $N_i=\Ker(\theta_i)$ and recall that $K(\calG_i)\leq N_i$.
It is straightforward to show that $\pi$ induces an epimorphism $N_1\to N_2$ of pro-$p$ groups.
By the functoriality of $K$, it further induces an epimorphism $N_1/K(\calG_1)\to N_2/K(\calG_2)$ of abelian pro-$p$ groups.

Moreover,  $\pi$ induces a group isomorphism $G_1/N_1\isom G_2/N_2\ (\leq\dbZ_p)$.
Therefore, and in view of  Proposition \ref{prop: exact sequence with M and N}, $\pi$ induces the following commutative diagram of elementary abelian $p$-groups:
\[
\xymatrix@C=0.6truecm{
1\ar[r] & N_1/K(\calG_1)N_1^p\ar[r]\ar[d] & G_1/G_1^p[G_1,G_1]\ar[r]\ar[d] & G_1/N_1G_1^p\ar[r]\ar[d] &1\\
1\ar[r] & N_2/K(\calG_2)N_2^p\ar[r] & G_2/G_2^p[G_2,G_2]\ar[r] & G_2/N_2G_2^p\ar[r] &1.\\
}
\]
Since $\pi$ is a cover, the middle vertical map is an isomorphism.
The right vertical map is an isomorphism since $G_i/N_iG_i^p$ is the Frattini quotient of $G_i/N_i$, $i=1,2$.
By the snake lemma, the left vertical map is also an isomorphism, as required.
\end{proof}

\begin{lem}
\label{lem: bar pi}
For $i=1,2$ let $\calG_i=(G_i,\theta_i)$ be a finitely generated torsion free cyclotomic pro-$p$ pair, and set $N_i=\Ker(\theta_i)$.
Assume that there are continuous isomorphisms
\[
N_1/K(\calG_1)N_1^p\isom N_2/K(\calG_2)N_2^p,
 \qquad
G_1/N_1G_1^p\isom G_2/N_2G_2^p.
\]
Assume further that $G_1$ is a free pro-$p$ group.
Then there is an epimorphism $\pi\colon G_1/K(\calG_1)\to G_2/K(\calG_2)$  which induces the above isomorphisms,  and maps $N_1/K(\calG_1)$ onto $N_2/K(\calG_2)$.
\end{lem}
\begin{proof}
For $i=1,2$, since $\calG_i$ is torsion free, $G_i/N_i$ is either $\dbZ_p$ or $\{1\}$, whence is a free pro-$p$ group.
Proposition \ref{prop: exact sequence with M and N} gives a split short exact sequence of elementary abelian $p$-groups
\[
\xymatrix@C=0.6truecm{
1\ar[r] &  N_i/K(\calG_i)N_i^p\ar[r] &  G_i/G_i^p[G_i,G_i]\ar[r] & G_i/N_iG_i^p\ar[r] & 1.}
\]
Thus
\[
G_i/G_i^p[G_i,G_i]\isom (N_i/K(\calG_i)N_i^p)\oplus(G_i/N_iG_i^p).
\]
Therefore the isomorphisms in the assumptions of the lemma combine to an isomorphism $\bar\pi$ which makes the following diagram commutative with exact rows:
\[ \xymatrix@C=0.6truecm{
1\ar[r] & N_1\ar[r]\ar@{->>}[d]  & G_1\ar[r]\ar@{->>}[d] & G_1/N_1\ar[r]\ar@{=}[d]&1 \\
1\ar[r] & N_1/K(\calG_1)\ar[r]\ar@{->>}[d] & G_1/K(\calG_1)\ar[r]\ar@{->>}[d]& G_1/N_1\ar[r]\ar@{->>}[d]&1 \\
1\ar[r] & N_1/K(\calG_1)N_1^p\ar[r]\ar[d]^{\wr} & G_1/G_1^p[G_1,G_1]\ar[r]\ar[d]^{\wr}_{\bar\pi}& G_1/N_1G_1^p\ar[r]\ar[d]^{\wr}&1 \\
1\ar[r] & N_2/K(\calG_2)N_2^p\ar[r] & G_2/G_2^p[G_2,G_2]\ar[r]& G_2/N_2G_2^p\ar[r]&1 \\
1\ar[r] & N_2/K(\calG_2)\ar[r]\ar@{->>}[u] & G_2/K(\calG_2)\ar[r]\ar@{->>}[u]& G_2/N_2\ar[r]\ar@{->>}[u]&1 \\
1\ar[r] & N_2\ar[r]\ar@{->>}[u]  & G_2\ar[r]\ar@{->>}[u] & G_2/N_2\ar[r]\ar@{=}[u]&1. \\
}
\]

Choose a minimal generating subset $\bar X'_1$ of $N_1/K(\calG_1)N_1^p$, as well as a subset $\bar X_1''$ of $G_1^p/G_1[G_1,G_1]$ which is mapped bijectively onto a minimal generating subset of  $G_1/N_1G_1^p$.
The sets $\bar X_1', \bar X_1''$ correspond under $\bar\pi$ to subsets  $\bar X'_2,\bar X_2''$ of $N_2/K(\calG_2)N_2^p$, $G_2/G_2^p[G_2,G_2]$, respectively, with analogous properties.
For $i=1,2$,  the union $\bar X_i=\bar X_i'\cupdot\bar X_i''$ is a minimal generating subset of $G_i/G_i^p[G_i,G_i]$.
We lift $\bar X_i',\bar X_i''$ to subsets $\hat X'_i,\hat X''_i$ of $N_i$, $G_i$, respectively.
By the Frattini argument, $\hat X_i=\hat X_i'\cupdot \hat X_i''$ is a minimal generating subset of $G_i$.

Also let $X_i'$ be the image of $\hat X_i'$ in $N_i/K(\calG_1)$.
A second application of the Frattini argument shows that $X'_i$ generates $N_i/K(\calG_1)$.

Since $G_1$ is free, there is a unique continuous homomorphism $\hat\pi\colon G_1\to G_2$ which maps $\hat X_1$ bijectively onto $\hat X_2$ under the above correspondences.
It induces the isomorphism $\bar\pi$ on the Frattini quotients, as well as a continuous homomorphism
$\pi\colon G_1/K(\calG_1)\to G_2/K(\calG_2)$.
Since $\hat X_2$ generates $G_2$, the homomorphism $\hat\pi$ is onto $G_2$, and therefore $\pi$ is onto $G_2/K(\calG_2)$.
Since $\pi$ maps $X_1'$ onto $X_2'$ we have  $\pi(N_1/K(\calG_1))=N_2/K(\calG_2)$.
\end{proof}

\begin{prop}
\label{prop:modulo K}
Let $\calS=(S,\hat\theta)$ be a torsion free  cyclotomic pro-$p$ pair with $S$ a finitely generated free pro-$p$ group.
Then $\calS$ is Kummerian.
\end{prop}
\begin{proof}
We abbreviate $\hat N=\Ker(\hat\theta)$.
Recall that $K(\calS)\leq\hat N$ and $\hat N/K(\calS)$ is abelian.
The quotient $S/\hat N$ is either $\dbZ_p$ or $\{1\}$, whence $K(\calS/\hat N)=\{1\}$ (see Example \ref{exam:K of Zp}).

Let $A$ be a free abelian pro-$p$ group of the same rank as $\hat N/K(\calS)$.
Then
\begin{equation}
\label{eq:first isomorphism}
\hat N/K(\calS)\hat N^p\isom A/A^p.
\end{equation}
Let $\calG=(G,\theta)=A\rtimes(\calS/\hat N)$, and note that $A=\Ker(\theta)$.
By  Lemma \ref{lem:K of an extension}, $K(\calG)=\{1\}$.
We have
\begin{equation}
\label{eq:second isomorphism}
S/\hat NS^p\isom G/AG^p.
\end{equation}
Lemma \ref{lem: bar pi} yields a continuous epimorphism $\pi\colon S/K(\calS)\to G$ which induces (\ref{eq:first isomorphism}) and (\ref{eq:second isomorphism}), and maps $\hat N/K(\calS)$ onto $A$.
Thus $\pi$ restricts to an epimorphism $\hat N/K(\calS)\to A$ which is an isomorphism on the Frattini quotients.
Since $A$ is a free abelian group, the latter epimorphism is necessarily an isomorphism (Lemma \ref{lem:Frattini quotients and isomorphisms}).
Thus $\hat N/K(\calS)$ is also a free abelian pro-$p$ group, as required.
\end{proof}

We now come to the main result of this section:

\begin{thm}
\label{thm:freeness equiv conditions}
Let $\calG=(G,\theta)$ be a finitely generated torsion free cyclotomic pro-$p$ pair.
The following conditions are equivalent.
\begin{itemize}
\item[(a)]
$\calG$ is Kummerian.
\item [(b)]
$\calG/K(\calG)=A\rtimes(\calG/\Ker(\theta))$ for a free abelian pro-$p$ group $A$.
\item[(c)]
The pro-$p$ group $\Ker(\theta)/K(\calG)$ is torsion free.
\item[(d)]
The pro-$p$ group $G/K(\calG)$ is torsion free.
\item[(e)]
Every cover $\calG'\to\calG$, with $\calG'$ Kummerian,
induces an isomorphism $\calG'/K(\calG')\to \calG/K(\calG)$.
\item[(f)]
There is a cover $\calS=(S,\hat\theta)\to\calG$, with $S$ a finitely generated free pro-$p$ group, such that the induced morphism $\calS/K(\calS)\to\calG/K(\calG)$ is an isomorphism.
\end{itemize}
\end{thm}
\begin{proof}
(a)$\Rightarrow$(b): \quad
This follows immediately from Proposition \ref{prop: calG/K}(a).

\medskip
(b)$\Rightarrow$(a), (b)$\Rightarrow$(c): \quad
We just note that $A=\Ker(\theta)/K(\calG)$.

\medskip

(c)$\Leftrightarrow$(d): \quad
Since $\calG$ is torsion free, $G/\Ker(\theta)\isom\Img(\theta)$ is a torsion free group.
The equivalence now follows from the semi-direct product decomposition $G/K(\calG)=(\Ker(\theta)/K(\calG))\rtimes(G/\Ker(\theta))$.

\medskip

(c)$\Rightarrow$(e): \quad
Set $\calG'=(G',\theta')$.
By the Frattini argument, $\calG'$ is also finitely generated.
By Lemma \ref{equal rank}, the cover $\calG'\to\calG$ induces an epimorphism
\begin{equation}
\label{eq: epimorphism}
\Ker(\theta')/K(\calG') \longrightarrow \Ker(\theta)/K(\calG)
\end{equation}
of abelian pro-$p$ groups, which an isomorphism on the Frattini quotients.
By assumption, $\Ker(\theta)/K(\calG)$ is torsion free,  and by Lemma \ref{lem:finitely generated extension}, it is finitely generated.
Hence it is a free abelian pro-$p$ group.
Lemma \ref{lem:Frattini quotients and isomorphisms} therefore implies that (\ref{eq: epimorphism}) is an isomorphism.
Also,
\[
G'/\Ker(\theta')\isom\Img(\theta')=\Img(\theta)\isom G/\Ker(\theta).
\]
A snake lemma argument now shows that the induced map $G'/K(\calG')\to G/K(\calG)$ is also an isomorphism,
and therefore the induced morphism $\calG'/K(\calG')\to\calG/K(\calG)$ is an isomorphism of cyclotomic pro-$p$ pairs.

\medskip

(e)$\Rightarrow$(f):
There is always a cover $\calS=(S,\hat\theta)\to\calG$, with $S$ a finitely generated free pro-$p$ group.
By Proposition \ref{prop:modulo K}, $\calS$ is Kummerian.

\medskip

(f)$\Rightarrow$(b): \quad
The pro-$p$ group $\Img(\hat\theta)=\Img(\theta)$ is torsion free, so $\calS$ is a torsion free cyclotomic pro-$p$ pair.
By Proposition \ref{prop:modulo K}, $\calS/K(\calS)\isom A\rtimes(\calS/\Ker(\hat\theta))$, with $A$ a free abelian pro-$p$ group.
\end{proof}


\section{1-cocycles}
Let $G$ be a pro-$p$ group and let $\theta\colon G\to\dbZ_p^\times$ be a continuous homomorphism.
It gives rise to an action of $G$ on $\dbZ_p$ by $g\alp=\theta(g)\alp$.
This action induces a $G$-action on $\dbZ/p^n$ for every $n\geq1$.
We denote the resulting $G$-modules by $\dbZ_p(1)_\theta$ and $\dbZ/p^n(1)_\theta$, respectively.

Recall that a continuous map $c\colon G\to \dbZ_p(1)_\theta$ is a {\sl 1-cocycle} if
\begin{equation}
\label{cocycle condition}
c(gh)=c(g)+\theta(g)c(h)
\end{equation}
for every $g,h\in G$.
In particular, $c$ is a continuous homomorphism on the commutator subgroup $[G,G]$.

The next lemma collects a few easy consequences of  (\ref{cocycle condition}).

\begin{lem}
\label{properties of c}
Let $\theta\colon G\to\dbZ_p^\times$ be a continuous homomorphism, let $c\colon G\to \dbZ_p(1)_\theta$ be a continuous 1-cocycle, and let $g,h\in G$.
Then:
\begin{enumerate}
\item[(a)]
$c(1)=0$;
\item[(b)]
$c(g\inv)=-\theta(g)\inv c(g)$;
\item[(c)]
$c\left(g^{-1}hg\right)=c(g^{-1})+\theta(g)^{-1}\left(c(h)+\theta(h)c(g)\right)$.
\item[(d)]
For every $\lam\in\dbZ_p$,
\[
c(g^\lam)=\begin{cases}
\lam c(g),& \hbox{ if }\theta(g)=1,\\
\dfrac{\theta(g)^\lam-1}{\theta(g)-1}c(g),&\hbox{ if } \theta(g)\neq1.
\end{cases}
\]
\item[(e)]
$c([g,h])=\theta(g\inv)\theta(h\inv)\bigl((1-\theta(h))c(g)-(1-\theta(g))c(h)\bigr)$.
\end{enumerate}
\end{lem}
\begin{proof}
(a), (b) and (c) follow directly from (\ref{cocycle condition}).

\medskip

(d) \quad
We first assume that $\lam=n$ is a non-negative integer.
Using (a) and (\ref{cocycle condition}) we obtain by induction that $c(g^n)=(\sum_{i=0}^{n-1}\theta(g)^i)c(g)$,
and the desired equality follows.

Next, for $\lam=-n$ a negative integer, (b) gives
$c(g^{-n})=-\theta(g)^{-n}c(g^n)$, and we use the previous case.

For an arbitrary $\lam\in\dbZ_p$ we use the density of $\dbZ$ in $\dbZ_p$ and a continuity argument.

\medskip

(e)\quad
By (\ref{cocycle condition}) and (b),
\[
\begin{split}
c([g,h])&= c(g\inv)+\theta(g\inv)\left(c(h\inv)+\theta(h\inv)\left(c(g)+\theta(g)c(h)\right)\right)\\
&= -\theta(g)\inv c(g)+\theta(g\inv)\left(-\theta(h)\inv c(h)+\theta(h\inv)\left(c(g)+\theta(g)c(h)\right)\right)\\
 &= -\theta(g)\inv c(g)- \theta(g\inv h\inv)c(h)+\theta(g\inv h\inv)c(g)+\theta(h\inv)c(h)\\
 &= \theta(g\inv)\theta(h\inv)\left((1-\theta(h))c(g)-(1-\theta(g))c(h)\right).
\end{split}
\]
\end{proof}

\begin{cor}
\label{cor:c vanishing on C}
Let $G$ be a profinite group, let $\theta\colon G\to\dbZ_p^\times$ be a continuous homomorphism, and let $c\colon G\to\dbZ_p(1)_\theta$ be a continuous 1-cocycle.
Then $c\inv(\{0\})$ is a closed subgroup of $G$.
\end{cor}
\begin{proof}
By the continuity, $c\inv(\{0\})$ is closed.
The cocycle condition (\ref{cocycle condition}) and Lemma \ref{properties of c}(a)(b) show that it is a subgroup of $G$.
\end{proof}

\begin{lem}
\label{lem:GK cocycleprop}
Let $\calG=(G,\theta)$ be a cyclotomic pro-$p$ pair.
For every continuous 1-cocycle $c\colon G\to\dbZ_p(1)_\theta$ one has $c(K(\calG))=\{0\}$.
\end{lem}
\begin{proof}
For $g\in G$ and $h\in \Ker(\theta)$ Lemma \ref{properties of c} gives
\[
\begin{split}
c(h^{-\theta(g)}ghg\inv)&=c(h^{-\theta(g)}) +c(ghg\inv)\\
&=-\theta(g)c(h)+c(g)+\theta(g)\left(c(h)+c(g\inv)\right) \\
&=c(g)+\theta(g)c(g\inv)=0.
\end{split}
\]
The claim now follows from Corollary \ref{cor:c vanishing on C}.
\end{proof}

Next let $G^{(i,p)}$, $i=1,2,\ldots,$ be the (pro-$p$) {\sl lower $p$-central series of $G$}, defined inductively by
\[
G^{(1,p)}=G, \quad   G^{(i+1,p)}=(G^{(i,p)})^p[G,G^{(i,p)}].
\]
Thus $G^{(i+1,p)}$ is the closed subgroup of $G$ generated by all elements $h^p$ and $[g,h]$,
where $g\in G$ and $h\in G^{(i,p)}$.

\begin{lem}
\label{lower central series}
Let  $\theta\colon G\to1+p\dbZ_p$ be a continuous homomorphism,
and let $c\colon G\to\dbZ_p(1)_\theta$ be a continuous 1-cocycle.
Then for every $i$,
\begin{itemize}
 \item[(a)]
 $\theta(G^{(i,p)})\subseteq 1+p^i\dbZ_p$;
 \item[(b)]
 $c(G^{(i,p)})\subseteq p^{i-1}\dbZ_p$.
\end{itemize}
\end{lem}
\begin{proof}
(a) \quad
By the binomial formula, $(1+p^i\dbZ_p)^p\subseteq1+p^{i+1}\dbZ_p$.
The assertion now follows by induction on $i$.

 \medskip

(b) \quad
We argue by induction on $i$.
For $i=1$ the claim is trivial.

Next let $i\geq1$, $g\in G$ and $h\in G^{(i,p)}$.
By induction $c(h)\in p^{i-1}\dbZ_p$, and by (a), $\theta(h)\in 1+p^i\dbZ_p$.
When $\theta(h)\neq1$ we have $(\theta(h)^p-1)/(\theta(h)-1)=\sum_{j=0}^{p-1}\theta(h)^j\in p\dbZ_p$.
We conclude from Lemma \ref{properties of c}(d)  that $c(h^p)\in p^i\dbZ_p$.

Further, by Lemma \ref{properties of c}(e),
\[
c([g,h])=\theta(g)\inv\theta(h)\inv\Bigl((1-\theta(h))c(g)-(1-\theta(g))c(h)\Bigr)\in p^i\dbZ_p.
\]

It remains to observe that, by (\ref{cocycle condition}) and Lemma \ref{properties of c}(a)(b),  $c^{-1}(p^i\dbZ_p)$ is a closed subgroup of $G$.
\end{proof}


\section{Kummerian pairs and 1-cocycles}
In the finitely generated torsion free case, we have the following cohomological characterization of  Kummerian pairs.

\begin{thm}
\label{thm:surjectivity criterion for Kummerian}
Let $\calG=(G,\theta)$ be a finitely generated torsion free cyclotomic pro-$p$ pair.
Then $\calG$ is Kummerian if and only if  the canonical map
\[
H^1(G,\dbZ/p^n(1)_\theta)\to H^1(G,\dbZ/p(1)_\theta)
\]
is surjective for every positive integer $n$.
\end{thm}
\begin{proof}
By Proposition \ref{prop: calG/K}(a), $\calG/K(\calG)=A\rtimes(\calG/\Ker(\theta))$, where $A=\Ker(\theta)/K(\calG)$ is an abelian pro-$p$ group.
By Lemma \ref{lem:finitely generated extension}, $A$ is a finitely generated pro-$p$ group.

We show that, for every $n\geq1$, the natural $G$-action on $H^1(A,\dbZ/p^n(1)_\theta)$ is trivial.
Indeed, for a 1-cocycle $c\colon A\to\dbZ/p^n(1)_\theta$, $g\in G$, and $h\in A$ one has $g\inv hg=h^{\theta(g\inv)}k$ for some $k\in K(\calG)$.
By Lemma \ref{lem:GK cocycleprop}, $c(k)=0$.
Using the cocycle condition (\ref{cocycle condition}) and Lemma \ref{properties of c}(d) we obtain that
\[
({}^gc)(h)=\theta(g)c(g\inv hg)=\theta(g)c(h^{\theta(g\inv)})=\theta(g)\theta(g\inv)c(h)=c(h).
\]
Therefore ${}^gc=c$ on $A$.
Consequently, $H^1(A,\dbZ/p^n(1)_\theta)^G=H^1(A,\dbZ/p^n)$.

Now $\calG$ is Kummerian if and only if  $A$ is a free abelian pro-$p$ group.
Since $A$ is finitely generated, this means that the map $H^1(A,\dbZ/p^n)\to H^1(A,\dbZ/p)$ is surjective for every $n\geq1$.

Let $p\dbZ/p^n(1)_\theta$ be the kernel of the $G$-module morphism $\dbZ/p^n(1)_\theta\to\dbZ/p$.
Since $G/\Ker(\theta)$ is either $\dbZ_p$ or $\{1\}$, the cohomology groups
\[
H^2(G/\Ker(\theta),\dbZ/p^n(1)_\theta), \quad
H^2(G/\Ker(\theta),\dbZ/p), \quad
H^2(G/\Ker(\theta),p\dbZ/p^n\dbZ(1)_\theta)
\]
are trivial \cite{NeukirchSchmidtWingberg}*{Prop.\ 3.5.17}.
Using the five term sequence we see that the above morphism induces a commutative diagram with exact rows
\[
\xymatrix{
0\ar[r] &H^1(G/\Ker(\theta),\dbZ/p^n(1)_\theta)\ar[r]\ar@{->>}[d] & H^1(G/K(\calG),\dbZ/p^n(1)_\theta)\ar[r]\ar[d] & H^1(A,\dbZ/p^n) \ar[r]\ar[d] & 0 \\
0\ar[r] &H^1(G/\Ker(\theta),\dbZ/p)\ar[r] & H^1(G/K(\calG),\dbZ/p)\ar[r] & H^1(A,\dbZ/p)  \ar[r] & 0, }
\]
where the left vertical map is surjective.
Therefore, by the snake lemma, the middle vertical map is surjective if and only if the right vertical map is surjective.
The assertion follows.
\end{proof}

\begin{rem}
\label{rem:relations of cyclotomic type for G_F}
\rm
Let $F$ be a field containing a root of unity of order $p$ (and containing $\sqrt{-1}$ if $p=2$).
Thus $\calG_F$ is finitely generated.
When $\calG_F$ is torsion free, Theorem \ref{thm:surjectivity criterion for Kummerian} gives an alternative proof of the fact that it is Kummerian  (Theorem \ref{thm:Kummer example}(f)).
Indeed, for  $\theta=\theta_{F,p}$ and $n\geq1$ we have a $G_F(p)$-module isomorphism $\dbZ/p^n(1)_\theta=\mu_{p^n}$.
Hence Kummer theory identifies the canonical homomorphism
\[
H^1(G_F(p),\dbZ/p^n(1)_\theta)\longrightarrow H^1(G_F(p),\dbZ/p(1)_\theta)
\]
with the projection $F^\times/(F^\times)^{p^n}\to F^\times/(F^\times)^p$, which is obviously surjective.
\end{rem}

The following equivalence is due to Labute \cite{Labute67a}*{Prop.\ 6}.

\begin{prop}
\label{Labute}
Let $\calG=(G,\theta)$ be a finitely generated cyclotomic pro-$p$ pair, and let $\bar X$ be a minimal system of generators of $G$.
The following conditions are equivalent:
\begin{enumerate}
\item[(a)]
For every $n\geq1$ the canonical map $\dbZ/p^n\to\dbZ/p$ induces an epimorphism
\[
H^1(G,\dbZ/p^n(1)_\theta)\longrightarrow H^1(G,\dbZ/p).
\]
\item[(b)]
For every $n\geq1$, every map $\bar\alp\colon \bar X\to\dbZ/p^n$ extends to a continuous 1-cocycle $c\colon G\to \dbZ/p^n(1)_\theta$.
\item[(c)]
Every map $\bar\alp\colon \bar X\to\dbZ_p$ extends to a continuous 1-cocycle $c\colon G\to \dbZ_p(1)_\theta$.
\end{enumerate}
\end{prop}

From this and from Theorem \ref{thm:surjectivity criterion for Kummerian} we deduce:

\begin{cor}
\label{cor:a to c}
Let $\calG=(G,\theta)$ be a finitely generated torsion free cyclotomic pro-$p$ pair.
Then $\calG$ is Kummerian if and only if conditions (a)--(c) of Proposition \ref{Labute} hold.
\end{cor}

Using Theorem \ref{thm:surjectivity criterion for Kummerian} we obtain additional examples of Kummerian pairs.

\begin{prop}
\label{prop:freeprod}
 Let $\calG_1$ and $\calG_2$ be cyclotomic pro-$p$ pairs.
Then  $\calG_1*\calG_2$ is  Kummerian if and only if both $\calG_1$ and $\calG_2$ are Kummerian.
\end{prop}
\begin{proof}
We write $\calG_i=(G_i,\theta_i)$, $i=1,2$, and $\calG_1*\calG_2=(G,\theta)$.
We may consider $\dbZ_p(1)_{\theta}$ and $\dbZ/p^n(1)_\theta$ as $G_i$-modules, $i=1,2$.
The $G_i$-modules $\dbZ/p(1)_{\theta_i}$, $i=1,2$, and the $G$-module $\dbZ/p(1)_\theta$ are trivial.

For every $n\geq1$ there is a commutative diagram
\[
\xymatrix{ H^1\left(G,\dbZ/p^n(1)_\theta\right)\ar[d]\ar[r]^-{\Res} &
H^1\left(G_1,\dbZ/p^n(1)_{\theta_1}\right)\oplus H^1\left(G_2,\dbZ/p^n(1)_{\theta_2}\right) \ar@<10ex>[d]\ar@<-10ex>[d] \\
H^1\left(G,\dbZ/p(1)_\theta\right)\ar[r]^-{\Res}& H^1\left(G_1,\dbZ/p(1)_{\theta_1}\right)\oplus H^1\left(G_2,\dbZ/p(1)_{\theta_2}\right) }.
\]
By \cite{NeukirchSchmidtWingberg}*{Th.\ 4.1.4, Th.\ 4.1.5}, the upper restriction map is surjective and the lower restriction map is an isomorphism.
Consequently, the left vertical map is surjective if and only if the middle and right vertical maps are surjective.
Now apply Theorem \ref{thm:surjectivity criterion for Kummerian}.
\end{proof}

As another important example, we recall that a pro-$p$ group $G$ is a {\sl Demu\v skin group} if the following conditions hold:
\begin{enumerate}
\item[(i)]
$H^1(G,\dbZ/p)$ is finite;
\item[(ii)]
$\dim_{\dbF_p} H^2(G,\dbZ/p)=1$;
\item[(iii)]
The cup product
\[
\cup\colon H^1(G,\dbZ/p)\times H^1(G,\dbZ/p)\to H^2(G,\dbZ/p)
\]
is non-degenerate.
\end{enumerate}

Note that, by (i), $G$ is finitely generated \cite{NeukirchSchmidtWingberg}*{Prop.\ 3.9.1}.
Explicit presentations of the Demu\v skin groups were given by Demu\v{s}kin \cite{Demushkin61}, Serre \cite{Serre63} and Labute \cite{Labute67a}.
If $F$ is a finite extension of $\dbQ_p(\mu_p)$, then $G_F(p)$ is Demu\v skin \cite{NeukirchSchmidtWingberg}*{Prop.\ 7.5.9}.
It is an open problem whether there are other Demu\v skin pro-$p$ groups (up to isomorphism) which are realizable as $G_F(p)$ for some field $F$ containing a root of unity of order $p$; Cf.\  \cite{Efrat03}.
The simplest example for which the problem is currently open appears to be the pro-$2$ group
\[
G=\langle x_1,x_2,x_3\ |\ x_1^2[x_2,x_3]=1\rangle.
\]
Cf.\ \cite{JacobWare89}*{Remark 5.5}.

By a result of Labute \cite{Labute67a}*{Th.\ 4}, for a Demu\v skin pro-$p$ group $G$ there is a unique continuous homomorphism $\theta\colon G\to 1+p\dbZ_p$ such that $\calG=(G,\theta)$ satisfies the surjectivity condition of Theorem \ref{thm:surjectivity criterion for Kummerian}.
We deduce:

\begin{thm}
\label{thm:demushkin}
For every torsion free Demu\v skin pro-$p$ group $G$ there is a unique continuous homomorphism $\theta\colon G\to1+p\dbZ_p$ such that the pair $(G,\theta)$ is Kummerian.
\end{thm}

We obtain the following additional characterization of Kummerian pairs.

\begin{thm}
\label{thm:K intersection}
Let $\calG=(G,\theta)$ be a finitely generated torsion free cyclotomic pro-$p$ pair.
Then $\calG$ is Kummerian if and only if
\[
K(\calG)=\bigcap_{c} c\inv(\{0\}),
\]
where $c$ ranges over all continuous 1-cocycles $G\to\dbZ_p(1)_\theta$.
\end{thm}
\begin{proof}
When $\theta=1$ we have $K(\calG)=[G,G]$ and the continuous 1-cocycles $c\colon G\to\dbZ_p(1)_\theta$ are exactly the continuous homomorphisms $G\to\dbZ_p$.
The equivalence therefore follows in this case from the structure of finitely generated torsion free pro-$p$ groups.
We may therefore assume that $\theta\neq1$.

Suppose that $K(\calG)=\bigcap_{c} c\inv(\{0\})$ with $c$ as above.
To prove that $\calG$ is Kummerian, it suffices to show that $\Ker(\theta)/K(\calG)$ is a torsion free pro-$p$ group (Theorem \ref{thm:freeness equiv conditions}).
To this end take $g\in\Ker(\theta)$ such that $g^n \in K(\calG)$ for some positive integer $n$.
For every 1-cocycle $c\colon G\to\dbZ_p(1)_\theta$ Lemma \ref{properties of c}(d) gives $0=c(g^n)=nc(g)$, whence $c(g)=0$.
By the assumption, $g\in K(\calG)$, as desired.

Conversely, suppose that $\calG$ is Kummerian.
By Corollary \ref{cor:a to c}, condition (c) of Proposition \ref{Labute} holds.
By  Lemma \ref{lem:GK cocycleprop}, $K(\calG)\subseteq\bigcap_cc\inv(\{0\})$.

For the converse inclusion, choose $x_0\in G$ whose coset generates $G/\Ker(\theta)\isom\dbZ_p$.
Also choose  $x_1\nek x_d\in G$ whose cosets modulo $K(\calG)$ form a minimal set of generators of  the free abelian pro-$p$ group $\Ker(\theta)/K(\calG)$.
Then $G=\langle x_0, x_1\nek x_d\rangle K(\calG)$.
Moreover, $x_0,\ldots,x_d$ form a minimal set of generators of $G$, as $K(\calG)\subseteq\Frat(G)$.

Given $g\in G$, we may write
\[
g=x_0^{\lambda_0}x_1^{\lambda_1}\cdots x_d^{\lambda_d}\cdot t
\]
for some $\lambda_0,\lambda_1\ldots,\lambda_d\in\dbZ_p$ and $t\in K(\calG)$.
Now assume further that $g\in\bigcap_cc\inv(\{0\})$.
Fix $0\leq i\leq d$ and set $h=x_0^{\lambda_0}\cdots x_{i-1}^{\lambda_{i-1}}$ and $h'=x_{i+1}^{\lambda_{i+1}}\cdots x_d^{\lambda_d}$.
Condition (c) of Proposition \ref{Labute} yields a continuous 1-cocycle $c_i\colon G\to\dbZ_p(1)_\theta$  such that  $c_i(x_i)=1$ and $c_i(x_j)=0$ for $j\neq i$.
By Lemma \ref{cor:c vanishing on C},
$c_i(h)=c_i(h')=0$, and by Lemma \ref{lem:GK cocycleprop}, $c_i(t)=0$.
Using Lemma \ref{properties of c} we compute:
\[
\begin{split}
0=c_i(g)&=c_i(hx_i^{\lam_i}h't)
=c_i(h)+\theta(h)\Bigl(c_i(x_i^{\lam_i})+\theta(x_i^{\lam_i})\bigl(c_i(h')+\theta(h')c_i(t)\bigr)\Bigr) \\
&= \theta(h)c_i(x_i^{\lam_i})
= \begin{cases}
\theta(h)\cdot\lam_i,& \hbox{ if }\theta(x_i)=1,\\
\theta(h)\cdot\dfrac{\theta(x_i)^{\lam_i}-1}{\theta(x_i)-1},
&\hbox{ if } \theta(x_i)\neq1.
\end{cases}
 \end{split}
 \]
As $\Img(\theta)$ is torsion free, this implies that $\lam_i=0$.
Since $i$ was arbitrary, $g=t\in K(\calS)$.
\end{proof}


\section{Groups which are not maximal pro-$p$ Galois groups}
\label{sec:not Galois}

In this section we provide examples of pro-$p$ groups $G$ which cannot be completed to a Kummerian torsion free cyclotomic pro-$p$ pair $(G,\theta)$.
Recall that when $p\neq2$, every cyclotomic pro-$p$ pair is torsion free.
Hence,  by Theorem \ref{thm:Kummer example}(f), in each of these examples, $G$ is not realizable as the maximal pro-$p$ Galois group of a field containing a root of unity of order $p$.
Similarly, when $p=2$, the group $G$ is not realizable as a maximal pro-$2$ Galois group of a field containing a root of unity of order $4$.

\begin{thm}
\label{thm:generalized AS}
Let $S$ be the free pro-$p$ group on the basis  $X=\{x_1,x_2,\ldots,x_d\}$ of $d$ elements, and $S_1$ its closed subgroup generated by $x_2,\ldots,x_d$.
Let $R$ be a closed normal subgroup of $S$ contained in $\Frat(S)$.
Suppose that $R$ contains a relation $r$ of one of the following types:
\begin{enumerate}
\item[(i)]
$r=x_1^\lambda s$, where $0\neq\lambda\in p\dbZ_p$ and $s\in S_1$;
\item[(ii)]
$r=x_1^\lambda s t$,
with $\lambda\in p\dbZ_p\setminus p^k\dbZ_p$ for some $k\geq2$, $s\in S_1\cap[S,S]$, and $t\in S^{(k+1,p)}\cap[S,S]$.
\end{enumerate}
Then $G=S/R$ cannot be completed into a Kummerian torsion free cyclotomic pro-$p$ pair $\calG=(G,\theta)$.
\end{thm}
\begin{proof}
Let $\theta\colon G\to 1+p\dbZ_p$ be  a continuous homomorphism with $\Img(\theta)$  torsion free.
Consider the map $\alp\colon X\to \dbZ_p$ given by $\alp(x_1)=1$, and $\alp(x_i)=0$, $i=2,\ldots, d$.
Suppose that $c\colon G\to\dbZ_p(1)_{\theta}$ is a continuous 1-cocycle extending $\alp$.
We denote the image of $u\in S$ in $G$ by $\bar u$.
Thus $\bar r=1$, so $c(\bar r)=0$.

Consider case (i).
Corollary \ref{cor:c vanishing on C} implies that $c(\bar s)=0$.
Now (\ref{cocycle condition}) and Lemma \ref{properties of c}(d) imply that
\[
\label{eq:generalized AS}
0=c(\bar r)=
\begin{cases}
\lambda,& \hbox{ if }\theta(\bar x_1)=1,\\
{\displaystyle \frac{\theta(\bar x_1)^\lambda-1}{\theta(\bar x_1)-1}},&\hbox{ if } \theta(\bar x_1)\neq1.
\end{cases}
\]
This contradicts the assumptions on $\lambda$ and $\Img(\theta)$.

Next we consider case (ii).
Since $\theta$ is trivial on $[G,G]$, we have $\theta(\bar s)=\theta(\bar t)=1$, whence $\theta(\bar x_1)=1$.
By Corollary \ref{cor:c vanishing on C}, $c(\bar s)=0$, and by Lemma \ref{lower central series}(b), $c(\bar t)\in p^k\dbZ_p$.
Using Lemma \ref{properties of c} we compute
\[
0=c(\bar r)=c(\bar x_1^\lambda)+c(\bar s\bar t)=\lambda+c(\bar s)+c(\bar t)=\lambda+c(\bar t),
\]
contrary to the assumptions on $\lambda$.
\end{proof}

A special case of Theorem \ref{thm:generalized AS} (in case (ii)) was earlier proved using other techniques by Rogelstad \cite{Rogelstad15}*{Th.\ 5.2.1} (who states in \cite{Rogelstad15}*{p.\ iii}  that it is a joint research with J.\ Min\'a\v c and N.D.\ T\^an).

\begin{rem}
\rm
When $p=2$, one cannot remove in Theorem \ref{thm:generalized AS} the condition that $\calG$ is torsion free.
For example, $\dbZ/2$ can be completed to the pro-$2$ cyclotomic pair $\calG_\dbR$, which is Kummerian, by Example \ref{thm:Kummer example}(f) (or by Example \ref{exam:Kummerian with trivial theta}(3)).
Another example is the pro-$2$ group
 \[
\langle x_1,x_2,x_3\mid x_1^2x_2^4[x_2,x_3]=1\rangle,
 \]
which is the underlying group of $\calG_{\dbQ_2}$ \cite{NeukirchSchmidtWingberg}*{Ch.\ VII, \S5, p.\ 417}.
\end{rem}

\begin{exam}
\label{exam:2relations}
\rm
Let $S$ be the free pro-$p$ group on the basis $X=\{x_1,x_2,x_3\}$ of 3 elements.
Let $0\neq \lambda_1,\lambda_2\in p\dbZ_p$ with $\lambda_1\neq\lambda_2$.
Let $R$ be a closed normal subgroup of $S$ contained in $\Frat(S)$ and containing the two relations
\[
r_1=x_1^{\lambda_1}[x_1,x_3]\qquad \text{and}\qquad r_2=x_2^{\lambda_2}[x_2,x_3].
\]
We show that $G=S/R$ cannot be completed into a Kummerian torsion free cyclotomic pro-$p$ pair $\calG=(G,\theta)$.

Indeed, let $\theta\colon G\to1+p\dbZ_p$ be a continuous homomorphism and suppose that $(G,\theta)$ is torsion free.
Set $\theta_i=\theta(\bar x_i)$ for $i=1,2,3$.
Since $\theta$ is trivial on commutators, $\theta_1^{\lambda_1}=\theta_2^{\lambda_2}=1$.
As $\Img(\theta)$ is torsion free,  $\theta_1=\theta_2=1$.
Now assume that $c\colon G\to\dbZ_p(1)_{\theta}$ is a continuous 1-cocycle such that $c(\bar x_1)=c(\bar x_2)=1$.
we may lift it to a 1-cocycle $\hat c\colon S\to\dbZ_p(1)_{\theta}$.
By Lemma \ref{properties of c}(d)(e),
$0=\hat c(r_i)=\lambda_i+\theta_3\inv-1$, $i=1,2$, contrary to $\lambda_1\neq \lambda_2$.
Therefore  $\calG$ is not Kummerian.
\end{exam}

For a pro-$p$ group $G$, let $H^n(G)=H^n(G,\dbZ/p)$ denote the $n$-th cohomology group of $G$,
with $\dbZ/p$ considered as a trivial $G$-module.
Then $H^*(G)=\bigoplus_{n\geq0}H^n(G)$ is a graded  $\dbF_p$-algebra with respect to the cup product.

We denote the (graded) tensor algebra of an abelian group $B$ by
\[{\rm Tens}_*(B)=\bigoplus_{n\geq0}{\rm Tens}_n(B).\]
Given a graded algebra $A_*=\bigoplus_{n\ge0}A_n$ we write $(A_*)_{\rm dec}$ for its {\sl decomposable part}, i.e., its subalgebra generated by $A_1$.
The graded algebra $A_*$ is called  {\sl quadratic} if the canonical morphism ${\rm Tens}_*(A_1)\to A_*$ is surjective and its kernel is generated by homogenous elements of degree $2$.
Notably, the Milnor $K$-ring  $K^M_*(F)$ of a field $F$ is quadratic.
Using the celebrated Rost--Voevodsky Theorem one obtains the following fundamental consequence (see, e.g.,  \cite{Quadrelli14}*{\S2} or \cite{CheboluEfratMinac12}*{Remark 8.2}):

\begin{thm}
\label{thm:BK}
Let $F$ be a field containing a root of unity of order $p$.
Then $H^*(G_F(p))$ is a quadratic $\dbF_p$-algebra.
\end{thm}

This fundamental fact allows us to show that the converse of Theorem \ref{thm:Kummer example}(f) does not hold in general, i.e., not all Kummerian cyclotomic pairs are realizable as $\calG_F$ for some field $F$ as above.

\begin{exam}
\label{exam:commutators1}
\rm
Let $S$ be the free pro-$p$ group on two generators, and let $G=S\times S$.
Then $G^{\rm ab}=\dbZ_p^4$, so by Example \ref{exam:Kummerian with trivial theta}(1), the cyclotomic pro-$p$ pair $\calG=(G,1)$ is Kummerian.

On the other hand, take a continuous epimorphism $\varphi\colon S\to \dbZ_p$, and let
\[
K=\{ (s,s')\in G=S\times S \mid \varphi(s)=\varphi(s') \}.
\]
It was shown in \cite{Quadrelli14}*{Th.\ 5.6} that $H^*(K)$ is not quadratic.
Consequently, $\calG$ is not realizable as $\calG_F$ for any field $F$ containing a root of unity of order $p$.
\end{exam}

The following result was proved in \cite{CheboluEfratMinac12} using the quadraticness
of the ring $H^*(G_F(p))$ (Theorem \ref{thm:BK}):

\begin{thm}
\label{CEM}
Let $F$ be a field containing a root of unity of order $p$ and let $G=G_F(p)$.
Then the inflation map $\Inf\colon H^*(G/G^{(3,p)})_{\rm dec}\to H^*(G)$ is an isomorphism.
\end{thm}

In \cite{EfratMinac16}, an even stronger version of this result is shown, in which $G^{(3,p)}$ is replaced by the third term of the {\sl  $p$-Zassenhaus filtration} of $G$.

\begin{exam}
\label{exam:commutators2}
\rm
Consider the pro-$p$ group
\[
G=\langle x_1,x_2,x_3\ |\ [[x_1,x_2],x_3]=1\rangle.
\]
The cyclotomic pro-$p$ pair $(G,1)$ is Kummerian, by Example \ref{exam:Kummerian with trivial theta}(1).

On the other hand, Theorem \ref{CEM} implies that $G$ is not realizable as the maximal pro-$p$ Galois group of any field containing a root of unity of order $p$.
Indeed, let $S$ be the free pro-$p$ group on 3 generators.
Then $S/S^{(3,p)}\isom G/G^{(3,p)}$, but $H^2(S)=0\neq H^2(G)$, since $G$ is not free pro-$p$ \cite{NeukirchSchmidtWingberg}*{Prop.\ 3.5.9}.
Theorem \ref{CEM} therefore implies that at least one of the groups $S$ and $G$ is not a Galois group as above.
But $S$ is well known to be realizable as a Galois group of this form, so $G$ is necessarily not (Cf.\ \cite{CheboluEfratMinac12}*{\S9} also for other examples of this type).
\end{exam}


\section{Connections with the triple Massey product criterion}
\label{section on Massey products}
The goal of this section is to give an example of a pro-$p$ group $G$ which can be ruled out from being a maximal pro-$p$ Galois group of a field containing a root of unity of order $p$ (resp., 4) when $p>2$ (resp., $p=2$) using the Kummerian property, but not using other known cohomological properties of such Galois groups: the Artin--Schreier/Becker restriction on finite subgroups,   quadraticness (Theorem \ref{thm:BK}),  and the 3-fold Massey product property for such groups (see below).
Thus  the Kummerian property appears to be a genuine new restriction on the structure of maximal pro-$p$ Galois groups.

Specifically, our example is $G=S/R$, where $S$ is the free pro-$p$ group on basis $x_1,\ldots, x_d$, with $d\geq3$ odd, and $R$ is its closed normal subgroup generated by
\begin{equation}
\label{special r}
 r  = x_1^{p^f}[x_2,x_3][x_4,x_5]\cdots[x_{d-1},x_d]
\end{equation}
for $f\geq1$.
When $p=2$ we further assume that $f\geq2$ (later on we will also require that $f\geq2$ when $p=3$).
It is a consequence of Theorem \ref{thm:generalized AS} that $G$ cannot be completed into a Kummerian torsion free cyclotomic pro-$p$ pair.
Hence it is not realizable as $G_F(p)$ for $F$ as above.
For $p\neq2$ and $d=3$, this was earlier shown in \cite{KochloukovaZalesskii05} using other methods.

We will use the connections between presentations of pro-$p$ groups using generators and relations on one hand and cup products and Bockstein elements on the other hand, as described e.g. in \cite{Labute67a}, \cite{NeukirchSchmidtWingberg}*{Ch.\ III, \S9}.
Let $S$ be a free pro-$p$ group, let $R$ be a closed normal subgroup of $S$, and let $G=S/R$.
There is a non-degenerate bilinear map
\[
(\cdot,\cdot)_R\colon \quad R/R^p[R,S]\times H^1(R)^S\to\dbF_p, \quad  (\bar g,\varphi)=-\varphi(g).
\]
In particular, $(\cdot,\cdot)_S$ gives a perfect duality between $S/\Frat(S)$ and $H^1(S)$.
When in addition $R$ is contained in $\Frat(S)$, the inflation $\Inf\colon H^1(G)\to H^1(S)$ is an isomorphism.
 As $H^2(S)=0$, the 5-term sequence implies that the transgression map
$\trg\colon H^1(R)^S\to H^2(G)$ is an isomorphism.
We obtain a non-degenerate bilinear map
\[
(\cdot,\cdot)'_R\colon \quad R/R^p[R,S]\times H^2(G)\to\dbF_p, \quad  (\bar g,\alp)'=(\trg\inv\alp)(g).
\]

Going back to our example, let $\chi_1,\ldots,\chi_d$ be the $\dbF_p$-linear basis of $H^1(G)=H^1(S)$
which is dual to the images $\bar x_1,\ldots,\bar x_d$ of $x_1,\ldots,x_d$ in $S/\Frat(S)$, with respect to $(\cdot,\cdot)_S$.
Let $\Bock_{G,p}$ be the {\sl Bockstein map}, i.e., the connecting homomorphism corresponding to the short exact sequence of $G$-modules
\[
0\to\dbZ/p\to\dbZ/p^2\to\dbZ/p\to0.
\]

\begin{prop}
\label{prop:cd2}
\begin{enumerate}
\item[(a)]
If $2\leq i\leq d$ or  $f\geq2$, then $(\bar r,\Bock_{G,p}(\chi_i))'_R=0$.
\item[(b)]
For $1\leq i\leq j\leq d$ one has $\chi_i\cup\chi_j\neq0$ if and only if $i$ is even and $j=i+1$.
\item[(c)]
For $\psi\in H^1(G)$ one has $\psi\cup H^1(G)=\{0\}$ if and only if $\psi=a\chi_1$ for some $a\in \dbF_p$.
\item[(d)]
$H^2(G)$ is one-dimensional.
\item[(e)]
$\cd(G)=2$.
\item[(f)]
The graded $\dbF_p$-algebra $H^*(G)$ is quadratic.
\item[(g)]
$G$ is torsion free.
\end{enumerate}
\end{prop}
\begin{proof}
Let $\bar r$ be the image of $r$ in $R/R^p[S,R]$.

\medskip

(a) \quad
By \cite{NeukirchSchmidtWingberg}*{Prop. 3.9.14}, $(\bar r,\Bock_{G,p}(\chi_i))'_R$ is $0$ for $2\leq i\leq d$, and is $p^{f-1}$ for $i=1$.
Further, $p^{f-1}=0\in\dbF_p$ for $f\geq2$.

\medskip

(b) \quad
By \cite{NeukirchSchmidtWingberg}*{Prop.\ 3.9.13}, one has
\[
(\bar r,\chi_2\cup\chi_3)'_R=(\bar r,\chi_4\cup\chi_5)'_R=\cdots=(\bar r,\chi_{d-1}\cup\chi_d)'_R=1,
\]
and $(\bar r,\chi_i\cup\chi_j)'_R=0$ for any other $i<j$.

Next suppose that $i=j$.
If $p>2$ then $\chi_i\cup\chi_i=0$, by the anti-symmetry of the cup product.
When $p=2$ we have $\chi_i\cup\chi_i=\Bock_{G,p}(\chi_i)$ \cite{EfratMinac11}*{Lemma 2.4} and by assumption, $f\geq2$.
Thus,  by (a), $(\bar r,\chi_i\cup\chi_i)'_R=(\bar r,\Bock_{G,p}(\chi_i))'_R=0$ in this case as well.

\medskip

(c) \quad
This follows from (b).

\medskip

(d) \quad
This follows from the fact that $G$ is a one-relator pro-$p$ group \cite{NeukirchSchmidtWingberg}*{Cor.\ 3.9.5}.

\medskip

(e) \quad
We use a method of Labute from \cite{Labute67b}.
Let $S^{(i)}$, $i=1,2,\ldots,$ be the (pro-$p$) lower central series of $S$, defined inductively by $S^{(1)}=S$ and $S^{(i+1)}=[S,S^{(i)}]$.
For $g\in S$ we define $\omega(g)=\sup\{i\mid g\in S^{(i)}\}$.
Set
\[
u= x_1, \ v=[x_2,x_3]\cdots[x_{d-1},x_d].
\]
Then $r=u^{p^f}v$, and one has $\omega(u)=1$ and $\omega(v)=2$.
By assumption, $f\ge 2$ when $p=2$, so
\[
\frac1f\left( f-1+\frac{\omega(v)}{\omega(u)}\right)=\frac{f+1}f<p.
\]
Therefore \cite{Labute67b}*{Th.\ 4} implies that $\cd(G)\leq2$.
By (b), $H^2(G)\neq0$.

\medskip

(f)\quad
This follows from  (d) and (e).

\medskip

(g) \quad
This follows from (e).
\end{proof}

\begin{rems}
\label{rem:mild}
\rm
(1) \quad
By Proposition \ref{prop:cd2}(g), $G$ cannot be ruled out from being a Galois group $G_F(p)$ as above by means of the Artin--Schreier/Becker restriction, namely that the nontrivial finite subgroups in such a group can only be of order $2$.
By Proposition \ref{prop:cd2}(f) it cannot be ruled out from being of the form $G_F(p)$ by means of the Voevodsky--Rost restriction, as in Theorem \ref{thm:BK}.

\medskip
(2) \quad
Since $\chi_1\cup H^1(G)=0$ (Proposition \ref{prop:cd2}(c)), $G$ is not a pro-$p$ Demu\v skin group.

\medskip

(3) \quad
By contrast, when $p=2$ and $f=1$, we have  by \cite{NeukirchSchmidtWingberg}*{Prop. 3.9.14}, $(\bar r,\Bock(\chi_1))'_R=1$, whence $\chi_1\cup\chi_1=\Bock(\chi_1)\neq0$.
Therefore $G$ is in this case a pro-$2$ Demu\v skin group.

\medskip

(4) \quad
When $p\neq2$, the fact that $\cd(G)=2$  also follows from Schmidt's \cite{Schimdt10}*{Th.\ 6.2}.
\end{rems}

We now explain the triple Massey product restriction on maximal pro-$p$ Galois groups.
Let $F$ be a field containing a root of unity of order $p$.
In addition to the ``internal" ring structure of $H^*(G_F(p))$ as a graded ring with the cup product, which is fully determined by the Voevodsky--Rost theorem (see \S\ref{sec:not Galois}), it carries an ``external" structure which can be used to rule out more groups from being maximal pro-$p$ Galois groups.
Specifically, there are known constrains on the {\it triple Massey product} in such groups.
Recall that for a pro-$p$ group $G$ and for $n\geq2$, the {\it $n$-fold Massey product} on $H^1(G)$ is a multi-valued map
\[
H^1(G)\times \cdots\times H^1(G)\longrightarrow H^2(G).
\]
For more details on this operation in the general homological context see \cite{Dwyer75}, \cite{Kraines66}.
See e.g.\  \cite{Efrat14}, \cite{EfratMatzri16}, \cite{MinacTan16b}, \cite{MinacTan16}, \cite{Morishita04}, \cite{Sharifi07}, \cite{Vogel04},  \cite{Wickelgren12a}, or \cite{Wickelgren12b} for  Massey products in the profinite and Galois-theoretic context.

As observed by Dwyer \cite{Dwyer75} for discrete groups (see also \cite{Efrat14}  for the profinite context)
Massey products $H^1(G)^n\to H^2(G)$ for a pro-$p$ group $G$ can be interpreted in terms of unipotent upper-triangular representations of $G$ as follows.
Let $I_{n+1}$ denote the $(n+1)\times(n+1)$ identity matrix
and let and $E_{i,j}$ be the $(n+1)\times(n+1)$ matrix with $1$ at entry $(i,j)$ and $0$ elsewhere.
Let $\dbU_{n+1}(\dbF_p)$ be the group of all unipotent upper-triangular $(n+1)\times(n+1)$-matrices over $\dbF_p$.
Its center $Z(\dbU_{n+1}(\dbF_p))$ consists of all matrices $I_{n+1}+aE_{1,n+1}$ with $a\in\dbF_p$.
Let \[\bar \dbU_{n+1}(\dbF_p)=\dbU_{n+1}(\dbF_p)/Z(\dbU_{n+1}(\dbF_p)).\]

\begin{lem}
\label{lem:massey representations}
Let $G$ be a pro-$p$ group and let $\varphi_1,\ldots,\varphi_n\in H^1(G)$, $n\geq2$.

\begin{enumerate}
\item[(a)]
The  Massey product $\langle\varphi_1,\ldots,\varphi_n\rangle$ is non-empty if and only if
there exists a continuous homomorphism $\gamma \colon G\to \bar\dbU_{n+1}(\dbF_p)$ such that
$\gamma_{i,i+1}=\varphi_i$ for $i=1,\ldots,n$.
\item[(b)]
The  Massey product $\langle\varphi_1,\ldots,\varphi_n\rangle$ contains $0$ if and only if
there exists a continuous homomorphism $\gamma \colon G\to \dbU_{n+1}(\dbF_p)$ such that
$\gamma_{i,i+1}=\varphi_i$ for $i=1,\ldots,n$.
\end{enumerate}
\end{lem}

Here $\gamma_{i,i+1}\colon G\to \dbF_p$ is the projection of $\gam$ on the $(i,i+1)$-entry.
Note that it is  a group homomorphism.
We call a Massey product $\langle\varphi_1,\ldots,\varphi_n\rangle$ {\sl essential} if it is non-empty, but does not contain $0$.

Next we focus on triple Massey products  $\langle\varphi_1,\varphi_2,\varphi_3\rangle$, with $\varphi_1,\varphi_2,\varphi_3\in H^1(G)$.
If this product is non-empty, then it is a coset of $\varphi_1\cup H^1(G)+\varphi_3\cup H^1(G)$ in $H^2(G)$ \cite{MinacTan16}*{Remark 2.2}.
We deduce:

\begin{lem}
\label{essential 3MP}
Suppose that $\dim_{\dbF_p}H^2(G)=1$.
If $\langle\varphi_1,\varphi_2,\varphi_3\rangle$ is essential, then it contains a single element, and $\varphi_1\cup H^1(G)=\varphi_3\cup H^1(G)=\{0\}$.
\end{lem}

In the Galois pro-$p$ context one has the following result of Matzri \cite{Matzri14}; see also \cite{EfratMatzri16} and \cite{MinacTan16b}.

\begin{thm}\label{thm:massey vanishes}
Let $F$ be a field containing a root of unity of order $p$.
Let $\varphi_1,\varphi_2,\varphi_3\in H^1(G_F(p))$.
Then $\langle\varphi_1,\varphi_2,\varphi_3\rangle$ is not essential.
\end{thm}

We now turn again to the group $G$ with the defining relation $r$ of (\ref{special r}), and assume that $p^f>3$.

\begin{prop}
\label{prop:massey}
Let $\varphi_1,\varphi_2,\varphi_3\in H^1(G)$.
Then the triple Massey product $\langle\varphi_1,\varphi_2,\varphi_3\rangle$ is not essential.
\end{prop}
\begin{proof}
As $\dim_{\dbF_p}H^2(G)=1$ (Proposition \ref{prop:cd2}(d)), and in view of Lemma \ref{essential 3MP}, we may assume that $\langle\varphi_1,\varphi_2,\varphi_3\rangle$ contains exactly one element and
$\varphi_1\cup H^1(G)=\varphi_3\cup H^1(G)=0$.
By Proposition \ref{prop:cd2}(c), $\varphi_1=a\chi_1,$ $\varphi_3=b\chi_1$ for some $a,b\in\dbF_p$.

Identifying $\varphi_1,\varphi_2,\varphi_3$ also as elements of $H^1(S)$, we define a continuous homomorphism $\hat\gamma\colon S\to\dbU_4(\dbF_p)$ by
\[
\hat\gamma(x) =
\begin{bmatrix}
1 & \varphi_1(x) & 0 & 0 \\ 0 & 1 & \varphi_2(x) & 0  \\
 0&0&1&\varphi_3(x) \\ 0&0&0&1 \end{bmatrix}.
 \]
Thus  $\hat\gamma(x_i) = I_4+\varphi_2(x_i)E_{2,3}$ for $2\leq i\leq d$.
Note that the map $\lambda\colon \dbF_p\to\dbU_4(\dbF_p)$,  $a\mapsto I_4+aE_{2,3}$, is a group homomorphism,
and there is a commutative square
\[
\xymatrix{
\langle x_2,\ldots,x_d\rangle\ar[r]^{\hat\gamma}\ar@{^{(}->}[d] &\dbU_4(\dbF_p)\\
G\ar[r]^{\varphi_2}&\dbF_p\ar[u]_{\lambda}.
}
\]
Since $\dbF_p$ is abelian, $\hat\gamma([x_i,x_{i+1}])= I_4$ for every $2\leq i<d$.
Also, since the characteristic is $p$, we have $\hat\gam(x_1^{p^f})=I_4+(\hat\gam(x_1)-I_4)^{p^f}$.
Since $p^f>3$, this gives $\hat\gam(x_1^{p^f})=I_4$, and therefore $\hat\gamma(r)=I_4$.
Thus $\hat\gamma$ induces a continuous homomorphism $\gamma\colon G\to\dbU_4(\dbF_p)$.
By Lemma \ref{lem:massey representations}(b), $\langle\varphi_1,\varphi_2,\varphi_3\rangle=\{0\}$.
\end{proof}

\begin{rem}
\rm
When $p=3$, the assumption in Proposition \ref{prop:massey} that $f\geq2$ cannot be omitted.
Indeed,  $\langle\chi_1,\chi_1,\chi_1\rangle=\{-\Bock_{G,3}(\chi_1)\}$ \cite{Vogel04}*{Prop.\ 1.2.15}.
When $f=1$ \cite{NeukirchSchmidtWingberg}*{Prop. 3.9.14} shows that $(\bar r,\Bock_{G,3}(\chi_1))'_R=1\in\dbF_3$.
This implies that $\Bock_{G,3}(\chi_1)\neq0$, so $\langle\chi_1,\chi_1,\chi_1\rangle$ is essential.
\end{rem}

\end{document}